\documentclass{amsart}
\usepackage{amsmath,amssymb,url}
\usepackage{graphics}
\usepackage{mathpazo}
\usepackage{mathrsfs}
\usepackage{calc}
\usepackage{pstricks}
\usepackage{cite}
\usepackage{url}

\DeclareSymbolFont{rsfscript}{OMS}{rsfs}{m}{n}
\DeclareSymbolFontAlphabet{\mathrsfs}{rsfscript}

 \newtheorem{Lemma}{Lemma}[section]
 \newtheorem{Theorem}{Theorem}[section]
 \newtheorem{Proposition}{Proposition}[section]
 \newtheorem{Corollary}{Corollary}[section]

\theoremstyle{definition}
 \newtheorem{Remark}{Remark}[section]
 \newtheorem{Example}{Example}[section]

\DeclareMathOperator{\var}{var}
\DeclareMathOperator{\Gr}{Gr}
\DeclareMathOperator{\rk}{rk}

\newcommand{\ovl}{\overline}
\newcommand{\wtil}{\widehat}

\newcommand{\gL}{\mathrsfs{L}}
\newcommand{\gR}{\mathrsfs{R}}
\newcommand{\gH}{\mathrsfs{H}}
\newcommand{\gD}{\mathrsfs{D}}
\newcommand{\gJ}{\mathrsfs{J}}
\newcommand{\K}{\mathrsfs{K}}

\newcommand{\V}{\mathbf{V}}
\newcommand{\W}{\mathbf{W}}
\newcommand{\B}{\mathbf{B}}
\newcommand{\CR}{\mathbf{CR}}

\begin{document}

\thispagestyle{empty}

\begin{flushleft}

\Large

\textbf{Local finiteness for Green's relations in semigroup varieties}

\end{flushleft}

\normalsize

\vspace{-2.25mm}

\noindent\rule{\textwidth}{1.5pt}

\vspace{2mm}

\begin{flushright}

\large
\textsc{Mikhail V. Volkov\footnote{Mikhail Volkov acknowledges support from the Russian Foundation for Basic Research, project no.\ 17-01-00551, the Ministry of Education and Science of the Russian Federation, project no.\ 1.3253.2017, and the Competitiveness Program of Ural Federal University.}}

\vspace{3pt}
\scriptsize
\textit{Institute of Natural Sciences and Mathematics,} \\
\textit{Ural Federal University,} \\
\textit{Lenina 51, 62000 Ekaterinburg, Russia} \\
\texttt{Mikhail.Volkov@usu.ru}

\vspace{2mm}

\large
\textsc{Pedro V. Silva\footnote{Pedro V. Silva was partially supported by CMUP (UID/MAT/00144/2013), which is funded by FCT (Portugal) with national (MEC) and European structural funds (FEDER), under the partnership agreement PT2020.}}

\vspace{3pt}
\scriptsize
\textit{Centro de Matem\'{a}tica, Faculdade de Ci\^{e}ncias, Universidade do Porto,} \\
\textit{Rua Campo Alegre 687, 4169-007 Porto, Portugal} \\
\texttt{pvsilva@fc.up.pt}

\vspace{2mm}

\large
\textsc{Filipa Soares\footnote{Filipa Soares was partially supported by CEMAT-Ci\^encias (UID/Multi/04621/2013), which is funded by FCT (Portugal) with national (MEC) and European structural funds (FEDER), under the partnership agreement PT2020.}}

\vspace{3pt}
\scriptsize
\textit{\'Area Departamental de Matem\'atica,} \\
\textit{ISEL --- Instituto Superior de Engenharia de Lisboa, Instituto Polit\'ecnico de Lisboa,} \\
\textit{Rua Conselheiro Em\'idio Navarro 1, 1959-007 Lisboa, Portugal} \\
\& \hspace{2mm} \textit{CEMAT--CI\^ENCIAS,} \\
\textit{Dep. Matem\'atica, Faculdade de Ci\^encias, Universidade de Lisboa} \\
\textit{Campo Grande, Edif\'icio C6, Piso 2, 1749-016 Lisboa, Portugal} \\
\texttt{falmeida@adm.isel.pt}

\end{flushright}

\vspace{5mm}

\normalsize

\section{Background and overview}
\label{sec:intro}

A semigroup $S$ is called \emph{locally finite} (respectively, \emph{periodic}) if each finitely generated (respectively, each monogenic) subsemigroup in $S$ is finite. A semigroup variety is \emph{locally finite} (respectively, \emph{periodic}) if so are all its members. Clearly, every locally finite variety is periodic but the converse is not true. The issues related to determining extra properties that distinguish locally finite semigroup varieties amongst periodic ones form a vast research area known as Burnside type problems. The reader can find a brief introduction into the main achievements in this area in~\cite[Chapter~3]{Sap14}; for the present discussion, it suffices to reproduce here just one powerful result by Sapir (a part of \cite[Theorem~P]{Sap87}).

Recall two notions involved in the formulation of Sapir's result. A variety is said to be of \emph{finite axiomatic rank} if for some fixed $n>0$, it can be given by a set of identities involving at most $n$ letters. (For instance, every variety defined by finitely many identities is of finite axiomatic rank.) A semigroup $S$ is said to be a \emph{nilsemigroup} if $S$ has an element 0 such that $s0=0s=0$ for every $s\in S$ and a power of each element in $S$ is equal to 0.

\begin{Proposition}
\label{prop:sapir}
A periodic semigroup variety $\V$ of finite axiomatic rank is locally finite if \textup(and, obviously, only if\textup) all groups and all nilsemigroups in $\V$ are locally finite.\qed
\end{Proposition}

It is easy to see that all groups in a periodic semigroup variety $\V$ form a semigroup variety themselves and so do all nilsemigroups in $\V$. Thus, Proposition~\ref{prop:sapir} completely reduces the problem of classifying locally finite semigroup varieties of finite axiomatic rank to two ``opposite'' special cases: the case of varieties consisting of periodic groups and the one of varieties consisting of nilsemigroups. (To emphasize the opposition observe that a periodic group could have been defined as a semigroup $S$ which has an element 1 such that $s1=1s=s$ for every $s\in S$ and a power of each element in $S$ is equal to 1. So, in a sense, nilsemigroups and periodic groups oppose to each other as 0 and 1.)

Sapir~\cite{Sap87} has found a deep and surprisingly elegant algorithmic characterization of locally finite varieties of finite axiomatic rank consisting of nilsemigroups. As for a classification of locally finite varieties of finite axiomatic rank consisting of periodic groups, this problem contains as a special case the classic Burnside problem which in the varietal language is nothing but the problem of classifying locally finite varieties among the Burnside varieties $\B_n$, $n=1,2,\dotsc$. (Recall that $\B_n$ consists of all groups of exponent dividing $n$.) This famous problem is commonly considered as being hopelessly difficult since infinite finitely generated groups of finite exponent belong to most complicated and mysterious objects of combinatorial algebra. The existence of such groups had remained unknown for a long time until it was finally established by Novikov and Adian (\!\cite{NoAd68}, see also \cite{Adi79}) for sufficiently large odd exponents and by Ivanov~\cite{Iva94} and independently by Lysenok~\cite{Lys96} for sufficiently large even exponents. Following \cite{Sap87}, we refer to infinite finitely generated groups of finite exponent as \emph{Novikov--Adian groups} in the sequel.

Having the above complication in mind, we intend to explore a natural relaxation of the property of being locally finite. First we recall that for a semigroup $S$, the notation $S^1$ stands for the least monoid containing $S$, that is\footnote{Here and throughout we use expressions like $A:=B$ to emphasize that $A$ is defined to be $B$.}, $S^1:=S$ if $S$ has an identity element and  $S^1:=S\cup\{1\}$ if $S$ has no identity element; in the latter case the multiplication in $S$ is extended to $S^1$ in a unique way such that the fresh symbol $1$ becomes the identity element in $S^1$.

The following five equivalence relations can be defined on every semigroup $S$:
\begin{itemize}
\item[] $x\,\mathrsfs{R}\,y \Leftrightarrow xS^1 = yS^1$, i.e., $x$ and $y$ generate the same right ideal;
\item[] $x\,\mathrsfs{L}\,y \Leftrightarrow S^1x = S^1y$, i.e., $x$ and $y$ generate the same left ideal;
\item[] $x\,\mathrsfs{J}\,y \Leftrightarrow S^1xS^1 = S^1yS^1$, i.e., $x$ and $y$ generate the same ideal;
\item[] $x\,\mathrsfs{H}\,y \Leftrightarrow xS^1 = yS^1 \land S^1x = S^1y$, i.e.,  $\mathrsfs{H} = \mathrsfs{R} \cap \mathrsfs{L}$;
\item[] $x\,\mathrsfs{D}\,y \Leftrightarrow (\exists z\in S)\ x\,\mathrsfs{R}\,z \land z\,\mathrsfs{L}\,y$, i.e., $\mathrsfs{D} = \mathrsfs{RL}$.
\end{itemize}
These relations, introduced by Green in 1951, cf. \cite{Gre51}, are collectively referred to as \emph{Green's relations}. They play a fundamental role in studying semigroups and, quoting Howie~\cite{How02}, are ``so all-pervading that, on encountering a new semigroup, almost the first question one asks is `What are the Green relations like?'"

Here we use Green's relations to introduce the following generalizations of local finiteness. Let $\K$ be one of the five Green relations $\gJ,\gD,\gL,\gR,\gH$. A semigroup $S$ is said to be \emph{$\K$-finite} if it has only finitely many $\K$-classes; otherwise, $S$ is said to be \emph{$\K$-infinite}. A semigroup variety $\V$ is \emph{locally $\K$-finite} if each finitely generated semigroup in $\V$ is $\K$-finite.

Clearly, locally finite varieties are locally $\K$-finite for each Green's relation $\K$. On the other hand, it is easy to see (cf.\ Lemma~\ref{lm:periodic} below) that  locally $\K$-finite varieties are periodic. Thus, locally $\K$-finite semigroup varieties occupy an intermediate position between locally finite and periodic varieties. The question of precisely locating this position appears to be quite natural within the general framework of Burnside type problems for semigroup varieties.

An additional a priori motivation for introducing local $\K$-finiteness came from the observation that, in every group, all Green's relations coincide with the universal relation whence all groups are locally $\K$-finite for each $\K$. Therefore, local $\K$-finiteness of a semigroup variety imposes no restriction on groups in this variety, and one could anticipate that this generalized local finiteness might ``bypass'' difficulties related to Novikov--Adian groups and lead to an ``absolute'' result rather than a characterization modulo groups in the flavor of Proposition~\ref{prop:sapir}. However, this tempting expectation has not been confirmed a posteriori: on the contrary, Novikov--Adian groups turn out to be behind our characterizations of locally $\K$-finite semigroup varieties. Namely, all the characterizations use the language of ``forbidden objects'' (when to a given property $\Theta$, say, of semigroup varieties, a certain list of ``forbidden'' semigroups is assigned so that a variety $\V$ satisfies $\Theta$ if and only if $\V$ excludes all members of the list), and for each $\K\in\{\gJ,\gD,\gL,\gR,\gH\}$, the ``forbidden objects'' corresponding to the property of being locally $\K$-finite are produced from Novikov--Adian groups by means of rather transparent constructions.

The paper is structured as follows. Section~\ref{sec:prelim} contains the necessary preliminaries. Section~\ref{sec:constructions} introduces four special families of infinite semigroups from which we later select our ``forbidden objects''. In Sections~\ref{sec:rlh} and~\ref{sec:dj} we provide characterizations of locally $\K$-finite varieties for $\K\in\{\gL,\gR,\gH\}$ and, respectively, $\K\in\{\gJ,\gD\}$. Section~\ref{sec:final} collects several concluding remarks and open questions.

\section{Preliminaries}
\label{sec:prelim}

We assume the reader's familiarity with some basic concepts of semigroup theory that can be found in the early chapters of any general
semigroup theory text such as, e.g., \cite{ClPr61,How95}. We also assume the knowledge of some rudiments of the theory of varieties such as the HSP-theorem and the concept of free objects; they all may be found, e.g., in~\cite[Chapter~II]{BuSa81}.

\subsection{Elementary facts about locally $\K$-finite varieties}
\label{subsec:elementary}

The Hasse diagram in Figure~\ref{fig:inclusions} shows inclusions between Green's relations that hold in every semigroup.
\begin{figure}[b]
\begin{center}
\unitlength=.60mm
  \begin{picture}(80,80)(10,0)
    \put(40,70){\circle*{3}}
    \put(42,72){$\mathrsfs{J}$}

    \put(40,50){\circle*{3}}
    \put(42,52){$\mathrsfs{D}$}

    \put(20,30){\circle*{3}}
    \put(10,32){$\mathrsfs{L}$}

    \put(60,30){\circle*{3}}
    \put(62,32){$\mathrsfs{R}$}

    \put(40,10){\circle*{3}}
    \put(40, 3){$\mathrsfs{H}$}

    \put(40,70){\line(0,-1){20}}
    \put(40,50){\line(-1,-1){20}}
    \put(40,50){\line(1,-1){20}}
    \put(20,30){\line(1,-1){20}}
    \put(60,30){\line(-1,-1){20}}
  \end{picture}
\caption{Inclusions between Green's relations}\label{fig:inclusions}
\end{center}
\end{figure}
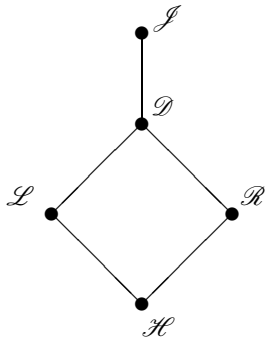
It is easy to see that if $\K_1,\K_2\in\{\gJ,\gD,\gL,\gR,\gH\}$ and $\K_1\subseteq\K_2$ in every semigroup, then each $\K_1$-finite semigroup is also $\K_2$-finite and each locally $\K_1$-finite variety is also locally $\K_2$-finite. From this and from the fact that $\mathrsfs{H} = \mathrsfs{R}\cap \mathrsfs{L}$ by definition, we immediately deduce the following relations between various versions of  $\K$-finiteness for individual semigroups and for semigroup varieties.

\begin{Lemma}
\label{lm:inclusions}
\emph{(i)} A semigroup is $\gH$-finite if and only if it is both $\gR$-finite and $\gL$-finite. A variety is locally $\gH$-finite if and only if it is both locally $\gR$-finite and locally $\gL$-finite.

\emph{(ii)} If a semigroup is either $\gR$-finite or $\gL$-finite, it is also $\gD$-finite. If a variety is either locally $\gR$-finite or locally $\gL$-finite, it is also locally $\gD$-finite.

\emph{(iii)} If a semigroup is $\gD$-finite, it is also $\gJ$-finite. If a variety is locally $\gD$-finite, it is also locally $\gJ$-finite.\qed
\end{Lemma}

The converse of either of the claims in (ii) is not true (see Example~\ref{ex:cr} in the next section). As for (iii), the converse of the claim for varieties holds true, see Corollary~\ref{cor:dj} below. On the level of individual semigroups, $\gD$-finiteness is \textbf{not} equivalent to $\gJ$-finiteness, even in the finitely generated case. An example of a finitely generated semigroup which is $\gJ$-finite while not $\gD$-finite can be found in~\cite[Example~3.3]{SiSo16}.

The ``right'' part of the following observation has been mentioned in the literature, see, e.g., \cite[Section 10, p.400]{Rus98}; its ``bilateral'' part can be obtained by a similar argument.

\begin{Lemma}
\label{lm:ideals}
A semigroup is $\gJ$-finite \textup($\gR$-finite, $\gL$-finite\textup) if and only if it has finitely many ideals \textup(respectively, right ideals, left ideals).\qed
\end{Lemma}

We need the following folklore characterization of periodic semigroup varieties:
\begin{Lemma}
\label{lm:periodic varieties}
For a semigroup variety $\V$, the following are equivalent:
\begin{itemize}
  \item[(i)] $\V$ is periodic;
  \item[(ii)] $\V$ excludes $\mathbb{N}$, the additive semigroup of positive integers;
  \item[(iii)] $\V$ satisfies an identity of the form $x^m=x^{m+k}$ for some $m,k>0$.\qed
\end{itemize}
\end{Lemma}

\begin{Lemma}
\label{lm:periodic}
For all $\K\in\{\gJ,\gD,\gL,\gR,\gH\}$, locally $\K$-finite varieties are periodic.
\end{Lemma}

\begin{proof}
If two integers $n,m\in\mathbb{N}$ are $\gJ$-related in $\mathbb{N}$, then clearly $n=m$. Thus, $\mathbb{N}$ is $\gJ$-infinite. Since $\mathbb{N}$ is generated by the number 1, locally $\gJ$-finite varieties must exclude $\mathbb{N}$, and by Lemma~\ref{lm:periodic varieties}, they are periodic. By Lemma~\ref{lm:inclusions} so are also locally $\K$-finite varieties for all $\K\in\{\gD,\gL,\gR,\gH\}$.
\end{proof}

Combining Lemma~\ref{lm:periodic} with the classic fact that $\gD=\gJ$ in every periodic semigroup (see, e.g., \cite{KoWa57}), we obtain

\begin{Corollary}
\label{cor:dj}
A variety is locally $\gD$-finite if and only if it is locally $\gJ$-finite.\qed
\end{Corollary}

The following fact also is an easy consequence of semigroup folklore but, since it plays an important role in this paper, we include its proof for the sake of completeness.
\begin{Lemma}
\label{lm:nil}
If $\K\in\{\gJ,\gD,\gL,\gR,\gH\}$ and $\V$ is a locally $\K$-finite variety, then nilsemigroups in $\V$ are locally finite.
\end{Lemma}

\begin{proof}
Let $S$ be a nilsemigroup and let $x,y\in S$ are $\gJ$-related. This means that $x=syt$ and $y=uxv$ for some $s,t,u,v\in S^1$. Then $x=suxvt$, and repeatedly substituting $suxvt$ for $x$ in the right hand side of the latter equality, we see that $x=(su)^nx(vt)^n$ for every $n>0$. If $x\ne y$, then either $su\in S$ or $vt\in S$, and since $S$ is a nilsemigroup,  a power of each element in $S$ is equal to 0. Thus, for $n$ large enough, either $(su)^n=0$ or $(vt)^n=0$, and in any case we have $x=0$ and then also $y=0$, a contradiction. Hence, $\gJ$-related elements of a nilsemigroup must be equal, and therefore, every infinite nilsemigroup is $\gJ$-infinite. From this and Lemma~\ref{lm:inclusions}, we obtain that
finitely generated nilsemigroups in $\V$ are finite.
\end{proof}

We say that a semigroup $S$ is a \emph{locally finite extension} of its ideal $J$ if the Rees quotient $S/J$ is locally finite and denote by $E(S)$ the set of all idempotents of $S$.

\begin{Lemma}
\label{lm:extension}
If $\K\in\{\gJ,\gD,\gL,\gR,\gH\}$ and $\V$ is a locally $\K$-finite variety, then every semigroup $S\in\V$ is a locally finite extension of any of its ideals containing $E(S)$.
\end{Lemma}

\begin{proof}
By Lemma~\ref{lm:periodic} $S$ is periodic whence each element in $S$ has an idempotent power. Hence the Rees quotient of $S$ over any ideal containing $E(S)$ is a nilsemigroup which must be locally finite by Lemma~\ref{lm:nil}.
\end{proof}

Combining Lemmas~\ref{lm:periodic} and~\ref{lm:nil} with Proposition~\ref{prop:sapir}, we immediately obtain

\begin{Corollary}
\label{cor:locfin}
Let $\K\in\{\gJ,\gD,\gL,\gR,\gH\}$ and let $\V$ be a variety of finite axiomatic rank that contains no Novikov--Adian groups. The variety $\V$ is locally $\K$-finite if and only if it is locally finite. \qed
\end{Corollary}

In view of Corollary~\ref{cor:locfin} our quest for locally $\K$-finite varieties should concentrate on varieties that contain Novikov--Adian groups.

\subsection{A property of Novikov--Adian groups}
\label{subsec:nag}

We will need the following property of semigroup varieties containing a Novikov--Adian group.

\begin{Lemma}
\label{lm:nag}
If a variety $\V$ contains a Novikov--Adian group, then $\V$ contains also a Novikov--Adian group $G$ such that some infinite finitely generated subgroup $H$ of $G$ and some $y\in G$ satisfy $H\cap y^{-1}Hy=\{1\}$.
\end{Lemma}

\begin{proof}
Let $G_0\in\V$ be a Novikov--Adian group with generators $g_1,\dots,g_n$. Denote by $F$ the group freely generated by the elements $x_0,x_1,\dots,x_n$ in the variety generated by $G_0$. Let $H_0$ be the subgroup of $F$ generated by $x_1,\dots,x_n$. Since the group $G_0$ is the image of $H_0$ under the homomorphism that extends the map $x_i\mapsto g_i$, $i=1,\dots,n$, the subgroup $H_0$ is infinite. Consider the subgroup $K:=H_0\cap x_0^{-1}H_0x_0$. If $h=h(x_1,\dots,x_n)$ is an arbitrary element in $K$, then $h$ can be also represented as $h=x_0^{-1}h'x_0$ for a suitable element $h'=h'(x_1,\dots,x_n)\in H_0$. For every $f\in F$, the map that fixes each free generator $x_i$, $i=1,\dots,n$, and sends $x_0$ to $f$ extends to an endomorphism of $F$. Applying this endomorphism to the equality $h=x_0^{-1}h'x_0$, we conclude that $h=f^{-1}h'f$ whence $fh=h'f$. In particular, if we choose the identity element of $F$ to play the role of $f$, we conclude that $h=h'$. Therefore, $fh=hf$ for every $f\in F$ which means that $h$ belongs to the center of $F$. Thus, the subgroup $K$ lies in the center of $F$; in particular, $K$ is Abelian and normal in $F$. Then the subgroup $H:=H_0/K$ of the quotient group $G:=F/K$ is infinite and finitely generated. (Indeed, assuming that $H$ is finite, we would derive from Schmidt's classic theorem \cite[p.429]{Rob96} that $H_0$ is finite as a finitely generated extension of the subgroup $K$, which is locally finite as an Abelian periodic group, by a finite group.) Clearly, if we let $y:=x_0K$, we get $H\cap y^{-1}Hy=\{1\}$.  Since the group $F$ lies in the variety $\V$, so does its quotient $G$.
\end{proof}

\subsection{Three 3-element semigroups}
\label{subsec:3-element}

The following semigroups defined on the set $\{e,a,0\}$ very often pop up in the literature on semigroup varieties:
\begin{itemize}
\item[] $N_2^1$, with $e^2=e$, $ea=ae=a$, and all other products equal $0$;
\item[] $N_2^\ell$, with $e^2=e$, $ea=a$, and all other products equal $0$;
\item[] $N_2^r$, with $e^2=e$, $ae=a$, and all other products equal $0$.
\end{itemize}
Each of these semigroups is obtained by adjoining an idempotent (denoted by $e$) to the 2-element zero semigroup $N_2:=\{a,0\}$;
the idempotent serves as a two-sided identity in $N_2^1$, as a left identity in $N_2^\ell$, and a right identity in $N_2^r$,
hence the notation.

Here we collect some known results involving the three semigroups that will be used in the present paper. Notice that in the sources to which we refer, notations for these semigroups vary; in particular, in \cite{PaVo05,PaVo06a,PaVo06b,Vol89,Vol00} the semigroup $N_2^1$ was denoted by $C$ and the semigroups $N_2^\ell$ and $N_2^r$ were denoted by $P$ and $P^*$ respectively.

By a \emph{divisor} of a semigroup $S$ we mean any homomorphic image of a subsemigroup of $S$, and by $\Gr S$ we denote the union of all subgroups of $S$. Observe that $\Gr S$ contains all idempotents of $S$. The following result is a plain semigroup version of \cite[Theorem~3.2]{Vol00}\footnote{There is a misprint in the formulation of this result in~\cite{Vol00}: ``left ideal'' should read ``right ideal''.}. (The paper \cite{Vol00} deals with epigroups but every periodic semigroup is an epigroup.) Since the proof of this result has been omitted in~\cite{Vol00}, we provide it here for the sake of completeness.

\begin{Lemma}
\label{lm:grSideal}
Let $S$ be a periodic semigroup. The set $\Gr S$ is a right ideal of $S$ if and only if none of the semigroups $N_2^1$ and $N_2^\ell$ are among the divisors of $S$.
\end{Lemma}

\begin{proof}
\emph{Necessity}. Each element $s\in S$ has an idempotent power which we denote by $s^\omega$, following the convention adopted in the finite semigroup theory. We let $s^{\omega+1}:=ss^\omega$. Then $s\in\Gr S$ if and only if $s=s^{\omega+1}$, and $s^{\omega+1}$ belongs to $\Gr S$ for every $s\in S$. These observations readily imply that the set $\Gr S$ forms a right ideal in $S$ if and only if
\begin{equation}
\label{eq:grSideal}
s^{\omega+1}t=(s^{\omega+1}t)^{\omega+1}
\end{equation}
for all $s,t\in S$. The condition \eqref{eq:grSideal} is inherited by the divisors of $S$; on the other hand, substituting $e$ for $s$ and $a$ for $t$, we see that \eqref{eq:grSideal} fails in both $N_2^1$ and $N_2^\ell$. Hence neither $N_2^1$ nor $N_2^\ell$ divides $S$.

\smallskip

\emph{Sufficiency}. Suppose that $\Gr S$ is not a right ideal of $S$. Then there exist $s,t\in S$ such that $s\in\Gr S$ and $st\notin\Gr S$. Let $e:=s^\omega$ and $a:=st$. Then $ea=a$. Consider the subsemigroup $T$ of $S$ generated by $e$ and $a$ and let $J:=T\setminus\{e,a\}$. It can be routinely checked that $J$ is an ideal in $T$. Consider the product $ae$. If $ae=e$, we multiply through by $a$ on the right and get $a^2=a$ which is impossible since $a$ is not a group element. Thus, either $ae=a$ or $ae\in J$. In the former case the Rees quotient $T/J$ is isomorphic to $N_2^1$ while in the latter one the quotient is isomorphic to $N_2^\ell$. Hence either $N_2^1$ or $N_2^\ell$ divides $S$.
\end{proof}

Combining Lemma~\ref{lm:grSideal} with its dual, we obtain a plain semigroup version of \cite[Corollary~3.3]{Vol00}:

\begin{Corollary}
\label{cor:grSideal}
Let $S$ be a periodic semigroup. The set $\Gr S$ is an ideal of $S$ if and only if none of the semigroups $N_2^1$, $N_2^\ell$, and $N_2^r$ are among the divisors of $S$.\qed
\end{Corollary}

Let $R_2$ stand for the 2-element right zero semigroup.

\begin{Lemma}[\cite{PaVo06a}, Lemma~2.8]
\label{lm:pdivisor}
The direct product $N_2^1\times R_2$ has the semigroup $N_2^\ell$ as a divisor.\qed
\end{Lemma}

Let $L_2$ denote the 2-element left zero semigroup. We say that a semigroup $S$ is a \emph{semigroup with central idempotents} if $es=se$ for all $e\in E(S)$ and $s\in S$.

\begin{Lemma}[\cite{Vol89}, Lemma~2]
\label{lm:central}
Every semigroup that has neither $N_2^\ell$ nor $N_2^r$ among its divisors and neither $R_2$ nor $L_2$ among its subsemigroups is a semigroup with central idempotents.\qed
\end{Lemma}

Combining Lemma~\ref{lm:pdivisor} and its dual with Lemma~\ref{lm:central}, we obtain

\begin{Corollary}
\label{cor:central}
If a semigroup variety $\V$ contains the semigroup $N_2^1$ but excludes the semigroups $N_2^\ell$ and $N_2^r$, then $\V$ consists of semigroups with central idempotents.\qed
\end{Corollary}

\subsection{Completely regular semigroups}
\label{subsec:cr}

Recall that a semigroup $S$ is called \emph{completely regular} is $S$ is a union of groups, in other words, if $S=\Gr S$. Here we state some classic facts about completely regular semigroups in the form convenient for the usage in the present paper. The first of these facts is basically Clifford's theorem (see \cite[Theorem~4.6]{ClPr61} or \cite[Theorem~4.1.3]{How95}). Observe that \emph{completely simple} semigroups can be defined as completely regular semigroups on which the relation $\gD$ coincides with the universal relation. By a \emph{semilattice} we mean a semigroup satisfying the identities $xy=yx$ and $x=x^2$.

\begin{Proposition} Let $S$ be a completely regular semigroup.

\emph{(i)} The relation $\gD$ on $S$ coincides with $\gJ$ and is a congruence on $S$.

\emph{(ii)} The quotient $S/\gD$ is a semilattice and the $\gD$-classes are completely simple semigroups \textup(called \emph{completely simple components} of $S$\textup).

\emph{(iii)} Green's relations $\gR$ and $\gL$ on $S$ are unions of the respective relations on the completely simple components of $S$.\qed
\label{prop:clifford}
\end{Proposition}

\begin{Corollary}
\label{cor:cr-ljf}
Every finitely generated completely regular semigroup is $\gJ$-finite.
\end{Corollary}

\begin{proof}
Let $S$ be a finitely generated completely regular semigroup. Then the semilattice $S/\gD$ is finitely generated and hence finite. Thus, $S$ has only finitely many $\gD$-classes, and since $\gD=\gJ$ on $S$, we see that $S$ is $\gJ$-finite.
\end{proof}

\begin{Proposition}
\label{prop:rees-suchkevich}
\emph{(i)} \emph{(Rees--Sushkevich's Theorem, see \cite[Theorem~3.5]{ClPr61} or \cite[Theorem~3.3.1]{How95})} For every completely simple semigroup $S$, there exist non-empty sets $I$ and $\Lambda$ \textup(the \emph{index sets}\textup), a group $G$ \textup(the \emph{structure group}\textup), and a $\Lambda\times I$-matrix $P=(p_{\lambda i})$ with entries in $G$ \textup(the \emph{sandwich matrix}\textup) such that $S$ is isomorphic to the \emph{Rees matrix semigroup} $M(G;I,\Lambda;P)$ defined as the set $I\times G\times\Lambda$ equipped with the multiplication
\[
(i,g,\lambda)(j,h,\mu):=(i,gp_{\lambda j}h,\mu).
\]

\emph{(ii)} \emph{(See, e.g., \cite[Lemma~3.2]{ClPr61})} Green's relations $\gR$ and $\gL$ on the Rees matrix semigroup $M(G;I,\Lambda;P)$ are characterized as follows:
\begin{align*}
(i,g,\lambda)\mathrel{\gR}(j,h,\mu)&\ \text{ if and only if }\ i=j;\\
(i,g,\lambda)\mathrel{\gL}(j,h,\mu)&\ \text{ if and only if }\ \lambda=\mu.&&\rule{5.6cm}{0pt}\qed
\end{align*}
\end{Proposition}

Even though representing completely simple semigroups as Rees matrix constructs has not been unavoidable for our proofs, we prefer to utilize Rees coordinates wherever this may straighten and simplify technicalities.

For every semigroup variety $\V$, we denote by $\CR(\V)$ the class of its completely regular members. We need also the following folklore fact.

\begin{Lemma}
\label{lm:crv}
If $\V$ is a periodic variety, the class $\CR(\V)$ also forms a variety.
\end{Lemma}

\begin{proof}
By Lemma~\ref{lm:periodic varieties}, $\V$ satisfies an identity of the form $x^m=x^{m+k}$ for some $m,k>0$. If a completely regular semigroup satisfies $x^m=x^{m+k}$, it satisfies also the identity $x=x^{k+1}$, and it is easy to see that every semigroup satisfying $x=x^{k+1}$ is completely regular. Therefore the class $\CR(\V)$ is exactly the subvariety of $\V$ defined within $\V$ by the identity $x=x^{k+1}$.
\end{proof}

\subsection{Finitely generated semigroups and small extensions}
\label{subsec:Rees}

Let $S$ be a semigroup and let $T$ be a subsemigroup of $S$. If the set $S\setminus T$ is finite, we say (following Jura~\cite{Jur78}) that  $T$ is a \emph{large subsemigroup} of $S$, and $S$ is a \emph{small extension} of $T$. It is known that many properties of semigroups are inherited by small extensions and by large subsemigroups, see \cite{CaMa13} for a comprehensive survey. We will need two results of this sort.

\begin{Lemma}[\cite{Jur78}]
\label{lm:Jura}
Every large subsemigroup of a finitely generated semigroup is finitely generated.\qed
\end{Lemma}

\begin{Lemma}[\cite{Rus98}, Sections 10 and 11]
\label{lm:small extension}
For each $\K\in\{\gL,\gR,\gJ\}$, every small extension of a $\K$-finite semigroup is $\K$-finite.\qed
\end{Lemma}

Since a semigroup is $\gH$-finite if and only if it is both $\gL$-finite and $\gR$-finite, Lemma~\ref{lm:small extension} implies a similar result for $\gH$-finiteness:

\begin{Corollary}
\label{cor:small extension}
Every small extension of a $\gH$-finite semigroup is $\gH$-finite.\qed
\end{Corollary}

We also need two further properties of finitely generated semigroups. The first of them comes from \cite[Theorem 8.2]{RRW98}. It gives a sufficient (in fact, also necessary) condition for the direct product of a finite semigroup with an infinite one to be finitely generated.

\begin{Lemma}
\label{lm:RRW}
Let $S$ be a finite semigroup with $S^2=S$. Then for every finitely generated semigroup $T$, the direct product $S\times T$  is finitely generated.\qed
\end{Lemma}

The second property is the following simple observation.
\begin{Lemma}
\label{lm:fg}
Let $S$ be a finitely generated semigroup and $T$ a subsemigroup of $S$ such that the complement $S\setminus T$ forms an ideal in $S$.
Then $T$ is finitely generated.
\end{Lemma}

\begin{proof}
Let $J:=S\setminus T$. If $X$ is a finite generating set of $S$, then every product of elements of $X$ involving a factor from $X\cap J$ must belong to $J$. Therefore every element of $T$ is a product of factors from $X\cap T$, that is, $X\cap T$ generates $T$.
\end{proof}

\section{Two constructions}
\label{sec:constructions}

As mentioned in Section~\ref{sec:intro}, our characterization of locally $\K$-finite varieties uses the language of ``forbidden objects''. In this subsection, we describe two constructions that yield our ``forbidden objects'' as special instances.

Our first construction is as follows. Let $G$ be a group and let $H$ be a subgroup in $G$. Denote by $L_H(G)$ the union of $G$ with the set $G_H=\{gH\mid g\in G\}$ of the left cosets of $H$ in $G$ and define a multiplication on $L_H(G)$ by keeping products in $G$ and letting for all $g_1,g_2\in G$,
\[
g_1(g_2H):=g_1g_2H\ \text{ and }\ (g_1H)g_2=(g_1H)(g_2H):=g_1H.
\]
Note that we view the coset $gH$ as different from $g$ even if $H$ is the trivial subgroup $E:=\{1\}$! It is easy to check that $L_H(G)$ becomes a semigroup in which $G$ is the group of units and $G_H$ is an ideal consisting of left zeros. In the dual way, for every group $G$ and its subgroup $H$, we define the semigroup $R_H(G)$ which is the union of $G$, being the group of units, and the set ${}_HG=\{Hg\mid g\in G\}$ of the right cosets of $H$ in $G$, being an ideal that consists of right zeros. Thus, the multiplication on $R_H(G)$ extends the multiplication in $G$ and is such that for all $g_1,g_2\in G$,
\[
(Hg_1)g_2:=Hg_1g_2\ \text{ and }\ g_1(Hg_2)=(Hg_1)(Hg_2):=Hg_2.
\]

Semigroups of the form $L_H(G)$ and $R_H(G)$ first appeared in~\cite{Ras81} for the special case where $G$ is a finite cyclic group and $H=E$; on the other hand, the semigroups are one-sided instances of a more general ``two-sided'' construction from~\cite[Section~1]{PaVo96}.

Our second construction looks similar to the first but in fact results in objects with different properties. Let $G$ be a group and let $H$ be a subgroup in $G$. Denote by $L^\flat_H(G)$ the union of the set $L_H(G)$ with the set $\{0\}$ where $0$ is a fresh symbol and define a multiplication on $L^\flat_H(G)$ by keeping products in $G$, letting
\[
 g_1(g_2H):=g_1g_2H\ \text{ and }\ (g_1H)g_2:=g_1H
\]
for all $g_1,g_2\in G$, and setting all other products equal to $0$. As above, we view the coset $gH$ as different from $g$ even if $H=E$. It is easy to check that $L^\flat_H(G)$ becomes a semigroup in which $G$ is the group of units and the union $G_H\cup\{0\}$ forms an ideal being a zero semigroup.

Observe that even though $L^\flat_H(G)$ is equal to $L_H(G)\cup\{0\}$ as a set, the semigroup $L^\flat_H(G)$ is \textbf{not} the same as the semigroup $L_H(G)$ with 0 adjoined. For instance, $L^\flat_H(G)$ has only two idempotents, namely, 0 and 1 (the identity element of the group $G$), in contrast to $L_H(G)$ with 0 adjoined since in the latter semigroup, besides 0 and 1, the whole set $G_H$ consists of idempotents.

In the dual way, for every group $G$ and its subgroup $H$, we define the semigroup $R^\flat_H(G)$ which is the union of $G$, being the group of units, the set ${}_HG$ of the right cosets of $H$ in $G$, and the singleton $\{0\}$, the union of the latter two being an ideal and a zero semigroup. The multiplication on $R^\flat_H(G)$ extends the multiplication in $G$ and is such that for all $g_1,g_2\in G$,
\[
(Hg_1)g_2:=Hg_1g_2\ \text{ and }\ g_1(Hg_2):=Hg_2,
\]
while all other products are equal to $0$.

Semigroups of the form $L^\flat_H(G)$ and $R^\flat_H(G)$ first appeared in~\cite{PaVo06a} (for the special case where $G$ is a finite cyclic group and $H=E$) under the names $L_1(G)$ and respectively $R_1(G)$.

We need a few properties of semigroups $L_H(G)$ and $L^\flat_H(G)$. The first of them readily follows from the constructions.

\begin{Lemma}
\label{lm:r-inf}
Each of the semigroups $L_H(G)$ and $L^\flat_H(G)$ is generated by any generating set of $G$ together with the coset $H$ and has the singletons $\{gH\}$, $g\in G$, as separate $\gR$-classes. In particular, if $G$ is a Novikov--Adian group and $H$ is its subgroup of infinite index, the semigroups $L_H(G)$ and $L^\flat_H(G)$ are finitely generated and $\gR$-infinite.\qed
\end{Lemma}

We denote the semigroups $L_E(G)$, $R_E(G)$, $L^\flat_E(G)$, and $R^\flat_E(G)$ where $E$ is the trivial subgroup by respectively $ L(G)$, $R(G)$, $L^\flat(G)$, and $R^\flat(G)$.

Now we are in a position to exhibit a locally $\gD$-finite variety which is neither locally $\gR$-finite nor locally $\gL$-finite and a finitely generated $\gD$-finite semigroup which is neither $\gR$-finite nor $\gL$-finite. (This is the example announced in Subsection~\ref{subsec:elementary}: it shows that the converse of either of the claims of Lemma~\ref{lm:inclusions}(ii) fails.)
\begin{Example}
\label{ex:cr}
Let $n\in\mathbb{N}$ be such that the Burnside variety $\B_n$ is not locally finite (for instance, any odd $n\ge 665$ will work, cf. \cite{Adi79}). Then the variety $\CR_n$ of the completely regular semigroups being unions of groups from $\B_n$ is locally $\gD$-finite by Corollary~\ref{cor:cr-ljf}. On the other hand, if $G\in\B_n$ is a Novikov--Adian group, then the semigroups $L(G)$ and $R(G)$ lie in $\CR_n$. By Lemma~\ref{lm:r-inf} $L(G)$ is finitely generated and $\gR$-infinite, and by the dual of this lemma $R(G)$ is finitely generated and $\gL$-infinite. Hence $\CR_n$ is neither locally $\gR$-finite nor locally $\gL$-finite. Moreover, the direct product $L(G)\times R(G)$ is easily seen to be a finitely generated $\gD$-finite semigroup which is neither $\gR$-finite nor $\gL$-finite.
\end{Example}

\begin{Lemma}
\label{lm:subgroups}
Let $G$ be a group and let $H$ and $K$ be subgroups in $G$ with $H\subseteq K$. Then $L_K(G)$ is a homomorphic image of $L_H(G)$ and $L^\flat_K(G)$ is a homomorphic image of $L^\flat_H(G)$.
\end{Lemma}

\begin{proof}
Define the map $\varphi_{H,K}\colon L_H(G)\to L_K(G)$ by letting, for every $g\in G$,
\[
\varphi_{H,K}(g):=g, \quad \varphi_{H,K}(gH):=gK.
\]
It is easy to see that $\varphi_{H,K}$ is a homomorphism of $L_H(G)$ onto $L_K(G)$. In the ``flat'' case, it suffices to extend $\varphi_{H,K}$ to $L^\flat_H(G)$ by letting $\varphi_{H,K}(0):=0$. The so extended map is then routinely verified to be a homomorphism of $L^\flat_H(G)$ onto $L^\flat_K(G)$.
\end{proof}

\begin{Lemma}
\label{lm:conjugate}
Let $G$ be a group and let $H$ and $K$ be conjugate subgroups in $G$. Then $L_H(G)$ and $L_K(G)$ are isomorphic, and so are  $L^\flat_H(G)$ and $L^\flat_K(G)$.
\end{Lemma}

\begin{proof}
Let $y\in G$ be such that $K=yHy^{-1}$. Define a map $\varphi\colon L_H(G)\to L_K(G)$ as follows: for all $g\in G$,
\begin{equation}
\label{eq:conjugate}
\varphi(g):=ygy^{-1},\quad \varphi(gH):=ygK.
\end{equation}
Clearly, $\varphi$ is a bijection. Take any $g_1,g_2\in G$. Then we have
\begin{gather*}
\varphi(g_1g_2)=yg_1g_2y^{-1}=yg_1y^{-1}yg_2y^{-1}=\varphi(g_1)\varphi(g_2),\\
\varphi(g_1(g_2H))=yg_1g_2K=yg_1y^{-1}(yg_2K)=\varphi(g_1)\varphi(g_2H),\\
\varphi((g_1H)g_2)=yg_1K=(yg_1K)yg_2y^{-1}=\varphi(g_1H)\varphi(g_2),\\
\varphi((g_1H)(g_2H))=yg_1K=(yg_1K)(yg_2K)=\varphi(g_1H)\varphi(g_2H).
\end{gather*}
Hence $\varphi$ is an isomorphism.

For the ``flat'' case, we define $\varphi\colon L^\flat_H(G)\to L^\flat_K(G)$ keeping the rules~\eqref{eq:conjugate} and adding the rule $\varphi(0):=0$. Again, $\varphi$ is a bijection, and the verification that $\varphi$ respects the multiplication works the same, except the last line that should be substituted by the following two:
\begin{gather*}
\varphi((g_1H)(g_2H))=\varphi(0)=0=(yg_1K)(yg_2K)=\varphi(g_1H)\varphi(g_2H),\\
\varphi(xy)=\varphi(0)=0=\varphi(x)\varphi(y)\ \text{ whenever $x=0$ or $y=0$}.
\end{gather*}
Hence $\varphi$ is an isomorphism also in this case.
\end{proof}

\begin{Lemma}
\label{lm:intersection}
Let $G$ be a group, let $H_i$, $i\in I$, be subgroups in $G$, and $H:=\bigcap_{i\in I}H_i$. Then the semigroup $L_H(G)$ is a subdirect product of the semigroups $L_{H_i}(G)$, $i\in I$, and the semigroup $L^\flat_H(G)$ is a subdirect product of the semigroups $L^\flat_{H_i}(G)$, $i\in I$.
\end{Lemma}

\begin{proof}
For each $i\in I$, let $\varphi_i:=\varphi_{H,H_i}$ where $\varphi_{H,H_i}\colon L_H(G)\to L_{H_i}(G)$ is the homomorphism of $L_H(G)$ onto $L_{H_i}(G)$ defined in the proof of Lemma~\ref{lm:subgroups}. In order to show that the homomorphisms $\varphi_i$ separate the elements of $L_H(G)$, it clearly suffices to verify that the homomorphisms separate any two different left cosets of $H$ in $G$. If $g_1H\ne g_2H$, we have $g_1^{-1}g_2\notin H$ whence $g_1^{-1}g_2\notin H_i$ for some $i\in I$. Then we get $g_1H_i\ne g_2H_i$, that is, $\varphi_i(g_1H)\ne\varphi_i(g_2H)$.

In the ``flat'' case, the same argument works.
\end{proof}

\begin{Lemma}
\label{lm:quotient}
Let $G$ be a group, let $N$ be a normal subgroup in $G$, and $\ovl{G}:=G/N$. Then the semigroup $L(\ovl{G})$ is a homomorphic image of the semigroup $L_N(G)$, and the semigroup $L^\flat(\ovl{G})$ is a homomorphic image of the semigroup $L^\flat_N(G)$.
\end{Lemma}

\begin{proof}
For $g\in G$ let $\ovl{g}$ stand for the image of $g$ under the natural homomorphism $G\to\ovl{G}$. We extend this homomorphism to a map $\varphi\colon L_{N}(G)\to L(\ovl{G})$ by letting $\varphi(gN)=\ovl{g}E$ where $E$ stands for the trivial subgroup of $\ovl{G}$. (Recall that we have adopted the convention that the coset $\ovl{g}E$ is viewed as different from $\ovl{g}$.) It is clear that $\varphi$ is onto, and it is easy to check that $\varphi$ is a homomorphism.

For the ``flat'' case, we only have to extend $\varphi$ by letting $\varphi(0):=0$.
\end{proof}

We are ready to prove a reduction result that simplifies our characterization of the locally $\K$-finite varieties. Given a semigroup $S$, let $\var S$ stand for the variety generated by $S$.

\begin{Proposition}
\label{prop:reduction}
Let $G$ be a Novikov--Adian group and let $H$ be its subgroup of infinite index. Then the variety $\var{L_{H}(G)}$ contains a semigroup of the form $L(\ovl{G})$ and the variety $\var{L^\flat_{H}(G)}$ contains a semigroup of the form $L^\flat(\ovl{G})$ where $\ovl{G}$ is a Novikov--Adian quotient of the group $G$.
\end{Proposition}

\begin{proof}
Let $N:=\bigcap_{y\in G}yHy^{-1}$. Obviously, the subgroup $N$ is normal and of infinite index in $G$. By Lemma~\ref{lm:intersection}, the semigroup $L_N(G)$ is a subdirect product of the semigroups $L_{yHy^{-1}}(G)$, $y\in G$, and by Lemma~\ref{lm:conjugate} each of these semigroups is isomorphic to the semigroup $L_{H}(G)$. Hence $L_N(G)$ belongs to the variety $\var{L_{H}(G)}$. If we denote the quotient group $G/N$ by $\ovl{G}$, then $\ovl{G}$ is a Novikov--Adian group and, by Lemma~\ref{lm:quotient}, the semigroup $L(\ovl{G})$ is a homomorphic image of $L_{N}(G)$ whence $L(\ovl{G})$ also belongs to $\var{L_{H}(G)}$. The same argument applies to the ``flat'' case.
\end{proof}

We formulate and prove the next result for semigroups of the form $L^\flat(G)$ only. (Its analogue for semigroups of the form $L(G)$ holds true but we do not need it.)

\begin{Lemma}
\label{lm:nag0}
If a variety $\V$ contains a semigroup of the form $L^\flat(G_0)$ where $G_0$ is a Novikov--Adian group, then $\V$ contains also a semigroup of the form $L^\flat(G)$ where $G$ is a Novikov--Adian group such that some infinite finitely generated subgroup $H$ of $G$ and some $y\in G$ satisfy $H\cap y^{-1}Hy=\{1\}$.
\end{Lemma}

\begin{proof}
Observe that, for any groups $G_1,G_2$, each onto homomorphism $\varphi\colon G_1\to G_2$ extends in an obvious way to an onto homomorphism $\ovl{\varphi}\colon L^\flat(G_1)\to L^\flat(G_2)$: we let $\ovl{\varphi}(g):=\varphi(g)$, $\ovl{\varphi}(gE):=\varphi(g)E$ for every $g\in G_1$, and $\ovl{\varphi}(0):=0$.

Now we follow the proof of Lemma~\ref{lm:nag}. Let the group $G_0$ be generated by $g_1,\dots,g_n$. Denote by $F$ the group freely generated by the elements $x_0,x_1,\dots,x_n$ in the variety generated by $G_0$. The group $F$ is known to be approximated by $G_0$ which means that there is a family $\Phi$ of homomorphisms from $F$ onto $G_0$ such that, for every pair $f_1,f_2$ of distinct elements in $F$, there exists a homomorphism $\varphi\in\Phi$ with $\varphi(f_1)\ne\varphi(f_2)$. Each homomorphism $\varphi\in\Phi$ can be extended to an onto homomorphism $\ovl{\varphi}\colon L^\flat(F)\to L^\flat(G_0)$, and it is clear that the homomorphisms $\ovl{\varphi}$ where $\varphi$ runs over $\Phi$ separate the elements of $L^\flat(F)$. Thus, $L^\flat(F)$ is a subdirect power of the semigroup $L^\flat(G_0)$, whence $L^\flat(F)$ lies in the variety $\V$. As shown in the proof of Lemma~\ref{lm:nag}, if $H_0$ is the subgroup of the group $F$ generated by $x_1,\dots,x_n$ and $K:=H_0\cap x_0^{-1}H_0x_0$, then $K$ is a normal subgroup of $F$, and the quotient group $G:=F/K$ is a Novikov--Adian group that possesses an infinite finitely generated subgroup $H$ (namely, $H_0/K$) and an element $y$ (namely, $x_0K$) satisfying $H\cap y^{-1}Hy=\{1\}$. The natural homomorphism $F\to G$ extends to an onto homomorphism $L^\flat(F)\to L^\flat(G)$, whence the semigroup $L^\flat(G)$ lies in $\V$.
\end{proof}

The next result, unlike the previous one, holds only for semigroups of the form $L^\flat(G)$ while its analogue for semigroups of the form $L(G)$ does not hold in general.

\begin{Lemma}
\label{lm:right-left}
For every group $G$, the semigroup $R^\flat(G)$ divides $L^\flat(G)\times L^\flat(G)$.
\end{Lemma}

\begin{proof}
We adapt the proof of \cite[Lemma 2.5]{PaVo05}.

Recall that $L^\flat(G)$ has as base set the union of $G$ with the set of the left cosets of the trivial subgroup $E:=\{1\}$ in $G$ and the singleton $\{0\}$. In the sequel we denote the left coset $gE$ by $\ovl{g}$ and the set of all these cosets by $\ovl{G}$. In this notation the multiplication in $L^\flat(G)$ becomes as follows: it extends the multiplication in $G$,
\[
g\ovl{h}:=\ovl{gh},\quad \ovl{g}h:=\ovl{g},
\]
for all $g,h\in G$, and all other products are equal to $0$.

Similarly, by denoting the right coset $Eg$ by $\wtil{g}$, we can describe the base set of $R^\flat(G)$ as $G \cup \{ \wtil{g} \mid g \in G \} \cup \{ 0 \}$. In this case the multiplication extends the multiplication in $G$,
\[
g\wtil{h}:=\wtil{h},\quad \wtil{g}h:=\wtil{gh},
\]
for all $g,h\in G$, and all other products are equal to $0$.

Let $T \subset L^\flat(G) \times L^\flat(G)$ be defined by
\[
T:= \{ (g,g) \mid g \in G \} \cup (\ovl{G} \times G) \cup (\{ 0\} \times G).
\]
It is straightforward to check that $T$ is a subsemigroup of $L^\flat(G)\times L^\flat(G)$.

Let the map $\varphi\colon T \to R^\flat(G)$ be defined by
\[
\varphi((g,g)):= g,\quad \varphi((\ovl{g},h)):= \wtil{g^{-1}h}, \quad \varphi((0,g)):= 0
\]
for $g,h \in G$. Obviously, $\varphi$ is onto. It it easy to check that $\varphi$ is a homomorphism. Indeed, we have
\begin{gather*}
\varphi((g,g)(\ovl{h},k)) = \varphi((\ovl{gh},gk)) = \wtil{h^{-1}k} = g\wtil{h^{-1}k} = \varphi((g,g))\varphi((\ovl{h},k)),\\
\varphi((\ovl{h},k)(g,g)) = \varphi((\ovl{h},kg)) = \wtil{h^{-1}kg} = \wtil{h^{-1}k}g = \varphi((\ovl{h},k))\varphi((g,g)),\\
\varphi((\ovl{g},h)(\ovl{k},\ell)) = \varphi((0,h\ell)) = 0 = \wtil{g^{-1}h}\wtil{k^{-1}\ell} = \varphi((\ovl{g},h))\varphi((\ovl{k},\ell))
\end{gather*}
for all $g,h,k,\ell \in G$, and the remaining cases are immediate. Hence $\varphi$ is an onto homomorphism whence $R^\flat(G)$ is a homomorphic image of $L^\flat(G)\times L^\flat(G)$.
\end{proof}

The following result now follows immediately from Lemma \ref{lm:right-left} and its dual.

\begin{Corollary}
\label{cor:same_variety}
The equality $\var{R^\flat(G)} = \var{L^\flat(G)}$ holds for every group $G$.\qed
\end{Corollary}

Note however that the equality $\var{R(G)} = \var{L(G)}$ fails even for groups with a very simple structure such as groups of prime order (see \cite{PaVo06a}).

The final result of this section compares varieties of the form $\var{L(G)}$ (or those of the form $\var{L^\flat(G)}$) when $G$ varies.

\begin{Lemma}
\label{lm:difference}
Let $G_1$ and $G_2$ be groups of finite exponent such that $\var G_1\ne\var G_2$. Then $\var L(G_1)\ne\var L(G_2)$ and $\var L^\flat(G_1)\ne\var L^\flat(G_2)$.
\end{Lemma}

\begin{proof}
Since $\var G_1\ne\var G_2$, there is an identity $u=v$ which holds in one of the groups and fails in the other. For certainty, let $u=v$ hold in $G_1$ and fail in $G_2$.
Denote by $n$ the least common multiple of the exponents of $G_1$ and $G_2$ and observe that $n>1$ since the group $G_2$ is non-trivial. The identity $x^n=1$ holds in both $G_1$ and $G_2$. Using this identity, one can rewrite the group words $u$ and $v$ such that they will contain only positive powers of their letters, that is, we may (and will) assume that $u$ and $v$ are semigroup words. Let $\{x_1,\dots,x_m\}$ be set of all letters that occur in either $u$ or $v$ and consider the identity
\begin{equation}
\label{eq:difference}
(x_1\cdots x_m)^nu=(x_1\cdots x_m)^nv.
\end{equation}
Clearly, this identity fails in each of the semigroups $L(G_2)$ and $L^\flat(G_2)$ as $G_2$ is a subgroup in each of them and \eqref{eq:difference} is equivalent to $u=v$ in every group of exponent dividing $n$. We claim that  \eqref{eq:difference} holds in both $L(G_1)$ and $L^\flat(G_1)$. Indeed, suppose that the letters $x_1,\dots,x_m$ are evaluated in one of these semigroups. If all their values belong to $G_1$, the values of $u$ and $v$ are equal, and hence, so are the values of $(x_1\cdots x_m)^nu$ and $(x_1\cdots x_m)^nv$. If some letter is evaluated at an element beyond $G_1$, the value of the word $(x_1\cdots x_m)^n$ will be a left zero in the case of $L(G_1)$ and 0 in the case of $L^\flat(G_1)$ --- the latter conclusion relies on the facts that $n>1$ and $L^\flat(G_1)\setminus G_1$ is an ideal being a zero semigroup. In either of the cases, the values of $(x_1\cdots x_m)^nu$ and $(x_1\cdots x_m)^nv$ coincide with that of $(x_1\cdots x_m)^n$. Thus, the identity \eqref{eq:difference} separates $\var L(G_1)$ from $\var L(G_2)$ and $\var L^\flat(G_1)$ from $\var L^\flat(G_2)$.
\end{proof}

\section{Locally $\K$-finite varieties for $\K\in\{\gL,\gR,\gH\}$}
\label{sec:rlh}

\subsection{Reduction to varieties of completely regular semigroups.} The first, relatively easy, step in our characterization of locally $\K$-finite varieties for $\K$ being one of the relations $\gL,\gR,\gH$ reduces the problem to the completely regular case.

\begin{Proposition}
\label{prop:c_out}
Let $\K\in\{\gL,\gR,\gH\}$ and let $\V$ be a locally $\K$-finite variety that contains a Novikov--Adian group. Then $\V$ excludes the semigroup $N_2^1$.
\end{Proposition}

\begin{proof}
Arguing by contradiction, assume that $\V$ contains the semigroup $N_2^1$. By Lemma~\ref{lm:nag} $\V$ contains a group $G$ such that some infinite finitely generated subgroup $H$ of $G$ and some $y\in G$ satisfy $H\cap y^{-1}Hy=\{1\}$. Let $g_1,\dots,g_n$ generate the subgroup $H$. Consider the direct product $N_2^1\times G\in\V$ and let $S$ be the subsemigroup of $N_2^1\times G$ generated by the pairs $(e,g_i)$, $i=1,\dots,n$, and $(a,y)$. It is easy to see that
\[
S=\left(\{e\}\times H\right)\cup\left(\{a\}\times HyH\right)\cup J,
\]
where $J$ stands for the set of all pairs from $S$ whose first entry is equal to 0. Suppose that two pairs $(a,h_1y)$ and $(a,h_2y)$, where $h_1,h_2\in H$, happen to be $\gR$-related in $S$. Then
\[
(a,h_1y)=(a,h_2y)(e,h_3)=(a,h_2yh_3)
\]
for some $h_3\in H$. Hence $h_1y=h_2yh_3$, and therefore, $y^{-1}h_2^{-1}h_1y=h_3$. Since $H\cap y^{-1}Hy=\{1\}$, we conclude that $h_3=1$ whence $h_1=h_2$. This means that $(a,h_1y)$ and $(a,h_2y)$ cannot be $\gR$-related in $S$ whenever $h_1\ne h_2$. We see that $S$ has at least $|H|$ different $\gR$-classes. Since the group $H$ is infinite, the variety $\V$ fails to be locally $\gR$-finite. Similarly, $S$ has at least $|H|$ different $\gL$-classes whence $\V$ is not locally $\gL$-finite as well. We have thus reached the desired contradiction.
\end{proof}

\begin{Proposition}
\label{prop:p_out}
Let $\K\in\{\gR,\gH\}$ and let $\V$ be a locally $\K$-finite variety that contains a Novikov--Adian group. Then $\V$ excludes the semigroup $N_2^\ell$.
\end{Proposition}

\begin{proof}
Let $G$ be a Novikov--Adian group in $\V$. Arguing by contradiction, assume that $\V$ contains the semigroup $N_2^\ell$. The direct product $S:=N_2^\ell\times G\in\V$ is finitely generated by Lemma~\ref{lm:RRW}. Since $(a,g)x\in\{0\}\times G$ for every $x\in S$, no pairs $(a,g)$ and $(a,h)$ with $g\ne h$ can be $\gR$-related in $S$. Thus, $S$ has at least $|G|$ different $\gR$-classes whence $\V$ is not locally $\gR$-finite, a contradiction.
\end{proof}

The following result gives the desired reduction to the completely regular case for the problem of classifying locally $\gH$-finite semigroup varieties provided they contain  Novikov--Adian groups.

\begin{Proposition}
\label{thm:hlf}
Let $\V$ be a variety that contains a Novikov--Adian group. The variety $\V$ is locally $\gH$-finite if and only if so is the variety $\CR(\V)$ and every semigroup in $\V$ is a locally finite extension of a periodic completely regular ideal.
\end{Proposition}

\begin{proof}
\emph{Necessity.} By Lemma~\ref{lm:periodic}, $\V$ is periodic. Thus, the variety $\CR(\V)$ is well defined by Lemma~\ref{lm:crv}; clearly, $\CR(\V)$ must be locally $\gH$-finite. Combining Proposition~\ref{prop:c_out}, Proposition~\ref{prop:p_out}, and the dual of Proposition~\ref{prop:p_out}, we see that $\V$ contains none of the semigroups $N_2^1$, $N_2^\ell$, and $N_2^r$. Now Corollary~\ref{cor:grSideal} implies that in every semigroup $S\in\V$, the set $\Gr S$ forms an ideal. The ideal $\Gr S$ is a periodic completely regular semigroup and contains all idempotents of $S$. Invoking Lemma~\ref{lm:extension}, we see that $S$ is a locally finite extension of a periodic completely regular ideal.

\smallskip

\emph{Sufficiency.} Let $S$ be a finitely generated semigroup in $\V$. Denote by $J$ a periodic completely regular ideal of $S$ such that the Rees quotient $S/J$ is locally finite. Since the semigroup $S/J$ is finitely generated, it is finite whence $J$ is a large subsemigroup of $S$. By Lemma~\ref{lm:Jura}, $J$ is a finitely generated semigroup. Since $J\in\CR(\V)$ and $\CR(\V)$ is locally $\gH$-finite, we conclude that $J$ is $\gH$-finite. Now Lemma~\ref{lm:small extension} applies, showing that $S$ also is $\gH$-finite. Hence $\V$ is locally $\gH$-finite.
\end{proof}

Similar (though somewhat more involved) reductions hold for locally $\gR$-finite and locally $\gL$-finite semigroup varieties. We formulate only the result for locally $\gR$-finite varieties as its version for locally $\gL$-finite varieties can be obtained by a straightforward dualization.

\begin{Proposition}
\label{thm:rlf}
Let $\V$ be a variety that contains a Novikov--Adian group. The variety $\V$ is locally $\gR$-finite if and only if so is the variety $\CR(\V)$ and every semigroup $S\in\V$ is a locally finite extension of an ideal of the form $SR$, where $R$ is a periodic completely regular right ideal in $S$.
\end{Proposition}

\begin{proof}
\emph{Necessity}. By Lemma~\ref{lm:periodic}, $\V$ is periodic whence the variety $\CR(\V)$ is well defined by Lemma~\ref{lm:crv}. Clearly, $\CR(\V)$ is locally
$\gR$-finite. Combining Propositions~\ref{prop:c_out} and~\ref{prop:p_out}, we see that $\V$ contains none of the semigroups $N_2^1$ and $N_2^\ell$. Lemma~\ref{lm:grSideal} ensures that in every semigroup $S\in\V$, the set $R:=\Gr S$ is a right ideal. Clearly, $R$ is periodic and completely regular and $SR$ is an ideal in $S$ which contains all idempotents of $S$. By Lemma~\ref{lm:extension} $S$ is a locally finite extension of $SR$.

\smallskip

\emph{Sufficiency}. Let $S$ be a finitely generated semigroup in $\V$ being a locally finite extension of its ideal $SR$ where $R$ is a periodic completely regular right ideal. Since the semigroup $S/SR$ is finitely generated, it is finite whence $S$ is a small extension of $SR$. By Lemma~\ref{lm:small extension}, if $SR$ is $\gR$-finite, then so is $S$. Thus, it suffices to show that $SR$ is an $\gR$-finite semigroup.

Lemma~\ref{lm:Jura} implies that the semigroup $SR$ is finitely generated. Let $s_1,\dots,s_n$ generate $SR$. There exist $r_1,\dots,r_n\in R$ such that for some $t_1,\dots,t_n\in S$ we have $s_i=t_ir_i$ for each $i=1,\dots,n$. Denote by $Q$ the subsemigroup of $R$ generated by $r_1,\dots,r_n$ and all products $r_it_j$ for $i,j=1,\dots,n$. Each $s=s_{i_1}s_{i_2}\cdots s_{i_k}\in SR$ can be rewritten as
\[
s=t_{i_1}(r_{i_1}t_{i_2})(r_{i_2}t_{i_3})\cdots (r_{i_{k-1}}t_{i_k})r_{i_k}\in t_{i_1}Q,
\]
whence we conclude that $SR\subseteq\cup_{i=1}^nt_iQ$. If $J$ is a right ideal in $SR$, we let
\[
J_i:=\{q\in Q\mid t_iq\in J\}.
\]
It is clear that
\begin{equation}
\label{eq:decomposition}
J=J\cap\bigl(\cup_{i=1}^nt_iQ\bigr)=\cup_{i=1}^n(J\cap t_iQ)=\cup_{i=1}^nt_iJ_i,
\end{equation} and it is easy to see that each $J_i$ forms a right ideal in $Q$. Since $Q\subseteq R$ is a completely regular semigroup in $\V$, it belongs to the variety $\CR(\V)$ which is locally $\gR$-finite, and since $Q$ is finitely generated, we conclude that $Q$ is $\gR$-finite. By Lemma~\ref{lm:ideals} $Q$ has only finitely many right ideals, and therefore, there are only finitely many possibilities to choose the right ideals $J_i$ of $Q$ in the decomposition~\eqref{eq:decomposition}. Hence, $SR$ has only finitely many right ideals and, by Lemma~\ref{lm:ideals}, $SR$ is $\gR$-finite.
\end{proof}

\subsection{Locally $\K$-finite varieties of periodic completely regular semigroups for $\K\in\{\gL,\gR,\gH\}$.} In view of Corollary~\ref{cor:locfin}, the reductions provided by Proposition~\ref{thm:hlf}, Proposition~\ref{thm:rlf} and the dual of the latter imply that, to complete the classification of locally $\K$-finite varieties of finite axiomatic rank with $\K\in\{\gH,\gR,\gL\}$, we have to characterize locally $\K$-finite varieties of periodic completely regular semigroups. Our next result provides characterizations of the latter varieties in the language of ``forbidden objects''. Observe that these characterizations do not require that varieties under consideration are of finite axiomatic rank.

\begin{Theorem}
\label{thm:cr}
Let $\V$ be a variety of periodic completely regular semigroups.

(r) $\V$ is locally $\gR$-finite if and only if $\V$ contains none of the semigroups $L(G)$ where $G$ is a Novikov--Adian group.

($\ell$) $\V$ is locally $\gL$-finite if and only if $\V$ contains none of the semigroups $R(G)$ where $G$ is a Novikov--Adian group.

(h) $\V$ is locally $\gH$-finite if and only if $\V$ contains none of the semigroups $L(G)$, $R(G)$ where $G$ is a Novikov--Adian group.
\end{Theorem}

\begin{proof}
\emph{Necessity} readily follows from Lemma~\ref{lm:r-inf} and its dual.

\smallskip

\emph{Sufficiency}. We are going to verify that if $\V$ contains a finitely generated $\gR$-infinite semigroup $S$, say, then $\V$ contains a semigroup of the form $L_H(G)$, where $G$ is a Novikov--Adian group and $H$ is its subgroup of infinite index. By Proposition~\ref{prop:reduction}, we conclude that $\V$ contains also a semigroup of the form $L(G)$, where $G$ is a Novikov--Adian group. This (by contraposition) will prove sufficiency in the claim (r). Sufficiency in the claim ($\ell$) then follows by duality, and sufficiency in the claim (h) follows too since every $\gH$-infinite semigroup is either $\gR$-infinite or $\gL$-infinite.

By Proposition~\ref{prop:clifford}(i) the relation $\gD$ is a congruence on $S$. Let $Y:=S/\gD$ and let $S_{\alpha}$, $\alpha\in Y$, be the completely simple components of $S$. The semilattice $Y$, being a homomorphic image of $S$, is finitely generated, and therefore, $Y$ is finite. We will induct on $|Y|$, but before we start the inductive proof, we need three preparatory steps. In the first two of them, we ``improve'' our semigroup $S$ keeping its key properties of being finitely generated and $\gR$-infinite.

\smallskip

\emph{Step 1}. Here we will show that it may be assumed that the \emph{kernel} (the least ideal) of $S$ is $\gR$-infinite. By Proposition~\ref{prop:clifford}(iii) there exists an $\gR$-infinite completely simple component of $S$. We order $Y$ with the usual semilattice order in which $\alpha\le\beta$ if and only if $\alpha\beta=\alpha$. Since $Y$ is finite, we can choose a maximal $\gamma\in Y$ with the property that $S_\gamma$ is $\gR$-infinite. Consider $T:=\cup_{\alpha\ge\gamma}S_\alpha$. It is easy to see that $T$ is a subsemigroup of $S$ while its complement $S\setminus T$ forms an ideal in $S$. By Lemma~\ref{lm:fg} $T$ is a finitely generated semigroup in $\V$ in which $S_\gamma$ is the least completely simple component, and hence, the kernel. Thus, without any loss, we can make the subsemigroup $T$ play the role of $S$ in the rest of the proof; in other words, we may assume that $\gamma\le\alpha$ for all $\alpha\in Y$ and $S_\gamma$ is $\gR$-infinite.

\smallskip

\emph{Step 2}. Now we aim to show that $S_\gamma$ may be assumed to consist of left zeros. Indeed, let $\rho$ be the relation on $S$ defined as the intersection of the Rees congruence corresponding to the ideal $S_\gamma$ and the relation $\gR$; in other words, for $s,t\in S$ we have $s\mathrel{\rho}t$ if and only if either $s=t$ or $s,t\in S_\gamma$ and $s\mathrel{\gR}t$. It is known and easy to verify that $\rho$ is a congruence on $S$; we include a verification for the sake of completeness. It suffices to verify that if $s,t\in S_\gamma$ and $s\mathrel{\gR}t$, then $sx\mathrel{\gR}tx$ and $xs\mathrel{\gR}xt$ for every $x\in S$. The latter relation holds because $\gR$ is known to be a left congruence on every semigroup, see, e.g., \cite[Section~2.1]{ClPr61} or \cite[Proposition~2.1.2]{How95}.

Let $M(G;I,\Lambda;P)$ be a Rees matrix semigroup isomorphic to $S_\gamma$. If $s,t\in S_\gamma$ and $s\mathrel{\gR}t$, then $s=(i,g,\lambda)$ and $t=(i,h,\mu)$ for some $i\in I$, $g,h\in G$, and $\lambda,\mu\in\Lambda$, see Proposition~\ref{prop:rees-suchkevich}(ii). Take $e:=(i,p^{-1}_{\lambda i},\lambda)\in S_\gamma$; we have $e^2=e$ and $se=s$. Similarly, if $f:=(i,p^{-1}_{\mu i},\mu)\in S_\gamma$, then $f^2=f$ and $tf=t$. For every $x\in S$ we have $sx=sex$ and $tx=tfx$. The products $ex$ and $fx$ belong to the ideal $S_\gamma$ whence $ex=e(ex)=(i,b,\nu)$ and $fx=f(fx)=(i,c,\kappa)$ for some $b,c\in G$ and $\nu,\kappa\in\Lambda$. Therefore we see that $sx=(i,gp_{\lambda i}b,\nu)$ and $tx=(i,hp_{\mu i}c,\kappa)$. This implies not only that $sx\mathrel{\gR}tx$ but also that $sx\mathrel{\gR}s$ for every $s\in S_\gamma$ and every $x\in S$.

The quotient semigroup $S/\rho$ is finitely generated and its kernel $S_\gamma/\rho$ consists of left zeros because $sx\mathrel{\rho}s$ for every $s\in S_\gamma$ and every $x\in S$ as we have just observed. Notice also that since $S_\gamma$ is $\gR$-infinite, so is $S_\gamma/\rho$. Therefore we can use $S/\rho$ in the role of $S$ in the sequel; in other words, we may assume that $S_\gamma$ is an infinite left zero semigroup.

\smallskip

\emph{Step 3}. Our final preparatory step deals with maximal (with respect to the order on $Y$) completely simple components of $S$. Let $S_\delta$ be such a maximal component. It is easy to see the complement $S\setminus S_\delta$ forms an ideal of $S$, and hence $S_\delta$ is finitely generated by Lemma~\ref{lm:fg}. Let $M(G;I,\Lambda;P)$ be a Rees matrix semigroup isomorphic to $S_\delta$ and let the set $X=\left\{(i_1,g_1,\lambda_1),\dots,(i_k,g_k,\lambda_k)\right\}$ generate $M(G;I,\Lambda;P)$. Since any product of triples from $X$ inherits the first coordinate from its first factor and the last coordinate from its last factor, we conclude that $I=\{i_1,\dots,i_k\}$ and $\Lambda=\{\lambda_1,\dots,\lambda_k\}$. In particular, the index sets $I$ and $\Lambda$ are finite and hence the sandwich matrix
$P=(p_{\lambda i})$ has only finitely many entries. The middle coordinate of any product of triples from $X$ is an alternating product of some elements from the set $Z:=\{g_1,\dots,g_k\}$ with the entries of $P$. Hence the structure group $G$ is generated by the finite set $Z\cup\{p_{\lambda i}\}$. Thus, each maximal completely simple component of $S$ is the union of finitely many copies of a finitely generated group.

\smallskip

We are ready to start the inductive proof. Recall that we consider a finitely generated semigroup $S\in\V$ being a semilattice of completely simple semigroups $S_{\alpha}$, $\alpha\in Y$, and such that its least completely simple component $S_\gamma$ is an infinite left zero semigroup. We aim to show, by induction on $|Y|$, that $\V$ contains a semigroup of the form $L_H(G)$, where $G$ is a Novikov--Adian group and $H$ is its subgroup of infinite index. Observe that the case $|Y|=1$ is impossible: indeed, in this case $S=S_\gamma$ but no infinite left zero semigroup can be finitely generated. Thus, the induction basis should be $|Y|=2$.

\smallskip

\emph{Basis.} Suppose that $|Y|=2$, that is, $Y=\{\gamma,\delta\}$ with $\gamma<\delta$. The component $S_\delta$ is then maximal and by Step~3 it is the union of finitely many copies of a finitely generated group $G$. If $G$ is finite, then $S_\delta$ is finite, and using Lemma~\ref{lm:Jura} we conclude that $S_\gamma$ is a finitely generated semigroup, a contradiction. Hence the group $G$ is infinite whence it is a Novikov--Adian group.

Fix a finite generating set $X$ of $S$ and represent an arbitrary element $s\in S_\gamma$ as a product of generators from $X$. Clearly,  such a product should contain a generator from the intersection $X\cap S_\gamma$. Let $z$ be the generator from $X\cap S_\gamma$ that occurs in the product first from the left. Taking into account that $z$ is a left zero in $S$, we see that $s\in S_\delta^1z$. Hence $S_\gamma=\cup_{z\in X\cap S_\gamma}S_\delta^1z$. The set $S_\gamma$ is infinite whence at least one of the finitely many sets of the form $S_\delta^1z$, where $z\in X\cap S_\gamma$, must be infinite. Fix an element $z\in X\cap S_\gamma$ such that $S_\delta^1z$ is infinite; the set $S_\delta z$ is then infinite too. Since $S_\delta$ is the union of finitely many copies $G_1,\dots,G_n$ of the group $G$, we conclude that the set $G_iz$ is infinite for some $i\in\{1,\dots,n\}$.

The set $T=G_i\cup G_iz$ is easily seen to be a subsemigroup of $S$ in which $G_iz$ is an infinite ideal consisting of left zeros. Let $e$ stand for the identity element of $G_i$ and let $H:=\{h\in G_i\mid hz=ez\}$. It is easy to see that $H$ is a subgroup of $G_i$. Given $g_1,g_2\in G_i$, the equality $g_1z=g_2z$ is equivalent to the inclusion $g_2^{-1}g_1\in H$ which in turn is equivalent to the equality $g_1H=g_2H$ between left cosets of $H$ in $G_i$. Hence the elements of $G_iz$ are in a 1-1 correspondence with the left cosets of $H$ in $G_i$ which implies that $H$ has infinite index in $G_i$. Combining the 1-1 correspondence with the identical map of $G_i$ onto itself, we obtain an isomorphism between $T$ and the semigroup $L_H(G_i)$. Thus, $L_H(G_i)\in\V$.

\smallskip

\emph{Inductive step}. Suppose that $|Y|>2$. Let $\delta$ be a maximal element of $Y$ and $S_\delta$ the corresponding completely simple component. If $S_\delta$ is finite, Lemma~\ref{lm:Jura} implies that the subsemigroup $J:=\cup_{\alpha\ne\delta}S_\alpha$ is finitely generated. Thus, $J$ is a finitely generated semigroup in $\V$ which is a semilattice of fewer completely simple semigroups and still has the infinite left zero semigroup $S_\gamma$ as its least completely simple component. Therefore we are in a position to apply the induction assumption to $J$.

Now suppose that the component $S_\delta$ is infinite. First consider the situation when for some $z\in S_\gamma$, the set $S_\delta z$ is infinite. Then the set $T:=S_\delta\cup S^1_\delta z$ is a subsemigroup in $S$ generated by $z$ together with any generating set of $S_\delta$. Since $S_\delta$ is finitely generated (see Step~3), $T$ is finitely generated too. The semigroup $T$ belongs to $\V$ and is a semilattice of the two completely simple semigroups $S_\delta$ and $S^1_\delta z$ of which the latter is an infinite left zero semigroup. Thus, the induction assumption applies to $T$.

Next consider the ``semi-infinite'' situation when all sets of the form $S_\delta z$, where $z\in S_\gamma$, are finite but there is no uniform upper bound on the sizes of these sets. This means that for every $n\in\mathbb{N}$, there exists an element $z_n\in S_\gamma$ with $|S_\delta z_n|>n$. Consider the direct power $D:=S^{\mathbb{N}}\in\V$ which elements we represent as functions $\mathbb{N}\to S$. Let $s_1,\dots,s_k$ generate $S_\delta$ (recall that $S_\delta$ is finitely generated, see Step~3) and denote by $\ovl s_i\in D$ the constant function taking value $s_i$ at each $n\in\mathbb{N}$. Clearly, the subsemigroup $\ovl S_\delta$ of $D$ generated by the functions $\ovl s_i$, $i=1,\dots,k$, is isomorphic to $S_\delta$. Consider the set $\ovl T:=\ovl S_\delta\cup\ovl S_\delta^1\ovl z$, where the function $\ovl z\in D$ is defined by the rule $\ovl z(n):=z_n$ for each $n\in\mathbb{N}$. Clearly, $\ovl T$ is a subsemigroup of $D$ whence $\ovl T$ belongs to $\V$. Further, it is easy to see that $\ovl S_\delta^1\ovl z$ is an infinite left zero semigroup and $\ovl T$ is a semilattice of the two completely simple semigroups $\ovl S_\delta$ and $\ovl S_\delta^1\ovl z$. Again, we can apply the induction assumption, this time to $\ovl T$.

Finally, consider the situation when there exists a number $n\in\mathbb{N}$ such that $|S_\delta z|\le n$ for all $z\in S_\gamma$. Let $s_1,\dots,s_k,t_1,\dots,t_m$ generate $S$ and be such that $s_1,\dots,s_k$ generate $S_\delta$ and $t_1,\dots,t_m$ belong to $J:=\cup_{\alpha\ne\delta}S_\alpha$. (We do not claim that $t_1,\dots,t_m$ generate $J$ nor that $J$ is finitely generated!) Fix $z\in S_\gamma$ and for each $s\in S_\delta$, consider a word $w$ of minimum length in the free monoid $\{x_1,\dots,x_k\}^*$ such that
$sz=w(s_1,\dots,s_k)z$. Take two decompositions $w=uv$ and $w=u'v'$ of the word $w$ in which $v$ is longer than $v'$ and assume that $v(s_1,\dots,s_k)z=v'(s_1,\dots,s_k)z$. Then we have
\[
sz=w(s_1,\dots,s_k)z=u(s_1,\dots,s_k)v(s_1,\dots,s_k)z=u(s_1,\dots,s_k)v'(s_1,\dots,s_k)z
\]
but the word $uv'$ is shorter than $w$ which contradicts the choice of $w$. Thus, all elements of the form $v(s_1,\dots,s_k)z$ where $v$ is a non-empty suffix of $w$ are different. The number of these elements is equal to the length of $w$ and they all belong to the set $S_\delta z$ whose cardinality does not exceed $n$. Thus, the length of the word $w$ also is less than or equal to $n$. The set $W_n$ of all words of length $\le n$ in the free monoid $\{x_1,\dots,x_k\}^*$ is finite. Hence the set $C:=\{w(s_1,\dots,s_k)\mid w\in W_n\}$
is finite too, and for every $z\in S_\gamma$ and every $s\in S_\delta$, there exists $c\in C$ such that $sz=cz$. We will refer to the latter conclusion as the \emph{contraction argument}.

If the generator $t_j$, $j=1,\dots,m$, lies in the completely simple component $S_{\alpha_j}$ and $\beta:=\delta\alpha_1\cdots\alpha_m$, then clearly the product $z_0:=s_1t_1t_2\cdots t_m$ belongs to $S_\beta$. Since the elements $\delta,\alpha_1,\dots,\alpha_m$ generate the semilattice $Y$, we have $\beta\gamma=\beta$, i.e., $\beta\le\gamma$. As $\gamma$ is the least element of $Y$, we conclude that $\beta=\gamma$, and thus, we have $z_0\in S_\gamma$.

For an arbitrary element $t\in S_\gamma$, there exists a word $u$ in the free monoid $\{x_1,\dots,x_k,y_1,\dots,y_m\}^*$ such that $t=u(s_1,\dots,s_k,t_1,\dots,t_m)$. Collecting factors from $S_\delta$, we can represent $t$ as an alternating product $p_1q_1\cdots p_{\ell-1}q_{\ell-1}p_\ell q_\ell$ where $p_1,\dots,p_\ell\in S_\delta$ ($p_1$ may be empty) and $q_1,\dots,q_\ell$ lie in the subsemigroup of $J$ generated by $t_1,\dots,t_m$ ($q_\ell$ may be empty). Since $t$ is a left zero, $t=tt_1z_0$. Applying the contraction argument to the elements $q_\ell t_1z_0\in S_\gamma$ and $p_\ell\in S_\delta$, we find an element $c_\ell\in C$ such that $p_\ell q_\ell t_1z_0=c_\ell q_\ell t_1z_0$. Then, applying the contraction argument to $q_{\ell-1}c_\ell q_\ell t_1z_0\in S_\gamma$ and $p_{\ell-1}\in S_\delta$, we find an element $c_{\ell-1}\in C$ such that $p_{\ell-1} q_{\ell-1}c_\ell q_\ell t_1z_0=c_{\ell-1}q_{\ell-1}c_\ell q_\ell t_1z_0$. Continuing in the same fashion, we find elements $c_{\ell-2},\dots,c_1\in C$ such that
\begin{align}
\label{eq:alternating}
t=tt_1z_0&=p_1q_1\cdots p_{\ell-1}q_{\ell-1}p_\ell q_\ell t_1z_0\notag\\
     &=p_1q_1\cdots p_{\ell-1}q_{\ell-1}c_\ell q_\ell t_1z_0\notag\\
     &=p_1q_1\cdots c_{\ell-1}q_{\ell-1}c_\ell q_\ell t_1z_0\\
     &\phantom{==}\hbox to 3.5 cm{\dotfill}\notag\\
     &=c_1q_1\cdots c_{\ell-1}q_{\ell-1}c_\ell q_\ell t_1z_0.\notag
\end{align}
Consider the set $\{ct_j,t_j\mid c\in C,\ j=1,\dots,m\}$. It is finite since the set $C$ is finite and is contained in $J$ since $t_j\in J$ and $J$ is an ideal in $S$. Hence the subsemigroup $U$ generated by this set is a finitely generated subsemigroup of $J$. We have
\[
U=U\cap J=U\cap\left(\cup_{\alpha\ne\delta}S_\alpha\right)=\cup_{\alpha\ne\delta}(U\cap S_\alpha),
\]
and each non-empty $U\cap S_\alpha$ is a completely simple semigroup since the class of periodic completely simple semigroups is closed under taking subsemigroups. Hence $U$ is a semilattice of fewer than $|Y|$ completely simple semigroups. The decomposition~\eqref{eq:alternating} ensures that every $t\in S_\gamma$ can be written as a product of factors of the form $ct_j$ with $c\in C$ and $t_j$, and hence, $t$ belongs to $U$. (Here one has to take into account that the generator $s_1$ belongs to $C$ as the value of the word $x_1$ of length~1, whence $z_0=s_1t_1t_2\cdots t_m$ can also be decomposed into a product of generators of $U$.) We see that $S_\gamma\subset U$ whence the least completely simple component of $U$ is an infinite left zero semigroup. Thus, $U$ satisfies all the conditions we need to apply the induction assumption. This completes the proof.
\end{proof}

\begin{Remark}
\label{rm:divisors}
Return for a moment to the situation named ``semi-infinite'' in the proof of Theorem~\ref{thm:cr}, see third paragraph in the inductive step. It is the situation when the chosen maximal completely simple component $S_\delta$ is infinite, all sets of the form $S_\delta z$, where $z\in S_\gamma$, are finite but there is no uniform upper bound on the sizes of these sets. By Step~3 $S_\delta$ is the union of finitely many copies of a finitely generated group $G$ which must be infinite and hence is a Novikov--Adian group. It is easy to verify that the condition that the sets $S_\delta z$ are finite for all $z\in S_\gamma$ but can be arbitrarily large implies that the group $G$ contains subgroups of arbitrarily large finite index. Since by a classical fact of group theory \cite[p.36]{Rob96} every subgroup of finite index contains a normal subgroup of finite index, $G$ contains normal subgroups of arbitrarily large finite index, and hence, $G$ has finite quotient groups of arbitrarily large order. Since each of these finite groups has no more generators than $G$ (as the groups are generated by the images of the generators of $G$) and has exponent dividing the exponent of $G$, we arrive at a negative answer to the question known as the \emph{Restricted Burnside Problem}: is the set of finite groups with a given number of generators and given exponent finite?

Zelmanov~\cite{Zel90,Zel91}, however, has answered the Restricted Burnside Problem in the affirmative for the case of any prime-power exponent. Long before that, Hall and Higman~\cite{HaHi56} proved a reduction theorem showing that a positive answer for the case of an arbitrary exponent follows from the prime-power case modulo the conjecture that there are only finitely many finite simple groups of given exponent. In turn, this conjecture is known to follow from the Classification of the Finite Simple Groups (CFSG), which at present is considered as having been completed, see, e.g., an overview in~\cite{Asc04}. We conclude that modulo CFSG the ``semi-infinite'' situation described above turns out to be impossible.

Why is this observation important in spite of the fact that we were able to handle the ``semi-infinite'' situation by quite elementary tools? If we analyze the above proof of sufficiency of Theorem~\ref{thm:cr}, we see that in all other cases we have proved more, namely, that the semigroup $S$ under consideration has a \textbf{divisor} of the form $L_H(G)$ where $G$ is a Novikov--Adian group and $H$ is its subgroup of infinite index. Thus, modulo CFSG our proof yields sufficiency in the following result which can be considered as a more precise form of Theorem~\ref{thm:cr}(r):

\begin{Corollary}
\label{cor:divisors}
A finitely generated semigroup $S$ belonging to a periodic completely regular variety is $\gR$-finite if and only if none of the semigroups $L_H(G)$ where $G$ is a Novikov--Adian group and $H$ is its subgroup of infinite index divide $S$.\qed
\end{Corollary}
Necessity of the condition of Corollary~\ref{cor:divisors} readily follows from the observation that the property of being $\gR$-finite is inherited by divisors within the class of periodic completely regular semigroups.

Of course, characterizations similar to Corollary~\ref{cor:divisors} can also be formulated for $\gL$-finite and $\gH$-finite completely regular semigroups.
\end{Remark}

\subsection{Characterizations.} Back to the problem of characterizing locally $\gH$-finite varieties of finite axiomatic rank, we are ready to state and to prove our ultimate result in this direction.

\begin{Theorem}
\label{thm:final_hlf}
A semigroup variety $\V$ of finite axiomatic rank is locally $\gH$-finite if and only $\V$ either is locally finite or it consists of locally finite extensions of some periodic completely regular ideal and contains none of the semigroups $L(G)$, $R(G)$ where $G$ is a Novikov--Adian group.
\end{Theorem}

\begin{proof}
\emph{Necessity}. If $\V$ excludes Novikov--Adian groups, then $\V$ is locally finite by Corollary~\ref{cor:locfin}. If $\V$ contains a Novikov--Adian group, Proposition~\ref{thm:hlf} applies, yielding that $\V$ consists of locally finite extensions of some periodic completely regular ideal. The fact that $\V$ contains none of the semigroups $L(G)$, $R(G)$ where $G$ is a Novikov--Adian group follows from Lemma~\ref{lm:r-inf} and its dual.

\smallskip

\emph{Sufficiency}. If $\V$ is locally finite, nothing is to prove. Suppose that $\V$ is not locally finite. Then $\V$ consists of locally finite extensions of a periodic completely regular ideal. In particular, $\V$ is periodic. By Lemma~\ref{lm:crv} the class $\CR(\V)$ of completely regular semigroups in $\V$ forms a variety, and this variety is locally $\gH$-finite by Theorem~\ref{thm:cr}(h). The only completely regular ideal of a nilsemigroup is $\{0\}$ whence every nilsemigroup being a locally finite extension of a completely regular ideal is locally finite. Thus, all nilsemigroups in $\V$ are locally finite, and Proposition~\ref{prop:sapir} implies that $\V$ must contain a Novikov--Adian group. Now Proposition~\ref{thm:hlf} guarantees that $\V$ is locally $\gH$-finite.
\end{proof}

The very same arguments, with Proposition~\ref{thm:rlf} and Theorem~\ref{thm:cr}(r) used instead of Proposition~\ref{thm:hlf} and respectively Theorem~\ref{thm:cr}(h), lead to the following characterization of locally $\gR$-finite semigroup varieties of finite axiomatic rank.

\begin{Theorem}
\label{thm:final_rlf}
A semigroup variety $\V$ of finite axiomatic rank is locally $\gR$-finite if and only if $\V$ either is locally finite or it consists of locally finite extensions of an ideal generated by a periodic completely regular right ideal and contains none of the semigroups $L(G)$, where $G$ is a Novikov--Adian group.\qed
\end{Theorem}

Of course, a characterization of locally $\gL$-finite semigroup varieties of finite axiomatic rank is given by the dual of Theorem~\ref{thm:final_rlf}.

\section{Locally $\K$-finite varieties for $\K\in\{\gD,\gJ\}$}
\label{sec:dj}

\subsection{Reduction to varieties of semigroups with central idempotents.} By Corollary~\ref{cor:dj} local $\gD$-finiteness and local $\gJ$-finiteness define the same class of varieties. Taking this into account, we will consider only locally $\gJ$-finite varieties in this section. Similarly to the trajectory of Section~\ref{sec:rlh}, our journey towards a characterization of locally $\gJ$-finite semigroup varieties of finite axiomatic rank starts with a reduction step, but this time we aim to reduce the problem to the case of varieties of semigroups with central idempotents.

\begin{Proposition}
\label{prop:pn_out}
Let $\V$ be a locally $\gJ$-finite variety that contains a Novikov--Adian group. Then at least one of the semigroups $N_2^\ell$ and $N_2^1$ does not belong to $\V$.
\end{Proposition}

\begin{proof}
Arguing by contradiction, assume that $\V$ contains both $N_2^\ell$ and $N_2^1$. Recall that in Subsection~\ref{subsec:3-element} we defined both these semigroups on the same set $\{e,a,0\}$; here in order to improve readability, we consider a copy of $N_2^1$ with the base set
$\{f,b,0\}$ in which $f^2=f$, $fb=bf=b$, and all other products equal 0. By Lemma~\ref{lm:nag} the variety $\V$ contains a group $G$ such that some infinite finitely generated subgroup $H$ of $G$ and some $y\in G$ satisfy $H\cap y^{-1}Hy=\{1\}$. Let the elements $g_1,\dots,g_n$ generate the subgroup $H$. Consider the direct product $N_2^\ell\times N_2^1\times G\in\V$ and let $S$ be its subsemigroup generated by the triples $(a,f,1)$, $(e,b,y)$, and $(e,f,g_i)$ for all $i=1,\dots,n$. It is not hard to verify that
\[
S=\left((e,f)\times H\right)\cup\left((a,f)\times H\right)\cup\left((e,b)\times HyH\right)\cup\left((a,b)\times HyH\right)\cup J,
\]
where $(e,f)\times H$ etc. serves as an abbreviation for $\{e\}\times\{f\}\times H$ etc., and $J$ stands for the set of all triples from $S$ whose first or second entries are equal to 0.

If two triples $(a,b,yh_1)$ and $(a,b,yh_2)$, where $h_1,h_2\in H$, are $\gJ$-related in $S$, then
\begin{equation}
\label{eq:division}
(a,b,yh_1)=p(a,b,yh_2)q\ \text{ for some }\ p,q\in S^1.
\end{equation}
Since $ax=0$ for every $x\in N_2^\ell$, we conclude that $(a,b,yh_2)t\in J$ for all $t\in S$ whence $q=1$ in~\eqref{eq:division}. Further, since $a^2=0$ in $N_2^\ell$ and $b^2=0$ in $N_2^1$, we see that $s(a,b,yh_2)\in J$ for every $s\in S\setminus\left((e,f)\times H\right)$. Hence $p=(e,f,h_3)$ for some $h_3\in H$. From~\eqref{eq:division} we deduce that $yh_1=h_3yh_2$, and therefore, $y^{-1}h_3y=h_1h_2^{-1}$. Since $H\cap y^{-1}Hy=\{1\}$, we conclude that $h_1h_2^{-1}=1$ whence $h_1=h_2$. This means that $(a,b,yh_1)$ and $(a,b,yh_2)$ cannot be $\gJ$-related in $S$ whenever $h_1\ne h_2$. We see that $S$ has at least $|H|$ different $\gJ$-classes. Since the subgroup $H$ is infinite, the variety $\V$ is not locally $\gJ$-finite, a contradiction.
\end{proof}

Recall that $N_2^r$ is the dual of the semigroup $N_2^\ell$. In the proof of Propositions~\ref{prop:ppi_out} and~\ref{prop:l(g)_out}, it is convenient to assume that $N_2^r$ has $\{f,b,0\}$ as the base set and $f^2=f$, $bf=b$, while all other products equal 0.

\begin{Proposition}
\label{prop:ppi_out}
Let $\V$ be a locally $\gJ$-finite variety that contains a Novikov--Adian group. Then at least one of the semigroups $N_2^\ell$ and $N_2^r$ does not belong to $\V$.
\end{Proposition}

\begin{proof}
Arguing by contradiction, assume that $\V$ contains both $N_2^\ell$ and $N_2^r$. Let $G$ be a Novikov--Adian group in $\V$. The direct product $S:=N_2^\ell\times N_2^r\times G\in\V$ is finitely generated by Lemma~\ref{lm:RRW}. Denote by $J$ the ideal of $S$ consisting of all triples whose first or second entries are equal to 0. As $ax=0$ for all $x\in P$ and $yb=0$ for all $y\in N_2^r$, we conclude that $s(a,b,g)t\in J$ for all $g\in G$ and all $s,t\in S$. Hence, no triples $(a,b,g)$ and $(a,b,h)$ with $g\ne h$ are $\gJ$-related in $S$. Thus, $S$ has at least $|G|$ different $\gJ$-classes whence $\V$ is not locally $\gJ$-finite, a contradiction.
\end{proof}

From Corollary~\ref{cor:locfin}, Proposition~\ref{prop:pn_out} and its dual, along with Proposition~\ref{prop:ppi_out}, we immediately obtain
\begin{Corollary}
\label{cor:1of3}
If a semigroup variety $\V$ of finite axiomatic rank is locally $\gJ$-finite, then either $\V$ is locally finite or it contains at most one of the semigroups $N_2^1$, $N_2^\ell$, and $N_2^r$.
\end{Corollary}

In our next result we meet again semigroups of the form $L(G)$ where $G$ is a group. Recall that $L(G)$ was introduced in Section~\ref{sec:constructions} as a sort of ``duplication'' of the group $G$: it is the union of $G$ with the set of the left cosets of the trivial subgroup $E:=\{1\}$ in $G$. Following the convention adopted in the proof of Lemma~\ref{lm:right-left} for the semigroup $L^\flat(G)$, we simplify our notation for the elements of $L(G)$ by denoting the coset $gE$ by $\ovl{g}$. In this notation the multiplication in $L(G)$ becomes as follows: it extends the multiplication in $G$ and for all $g_1,g_2\in G$,
\[
g_1\ovl{g_2}:=\ovl{g_1g_2}\ \text{ and }\ \ovl{g_1}g_2=\ovl{g_1}\ovl{g_2}:=\ovl{g_1}.
\]

\begin{Proposition}
\label{prop:l(g)_out}
Let a locally $\gJ$-finite variety $\V$ contain the semigroup $N_2^r$. Then $\V$ contains none of the semigroups $L(G)$ with $G$ being a Novikov--Adian group.
\end{Proposition}

\begin{proof}
Arguing by contradiction, assume that $\V$ contains some $L(G)$ where $G$ is a Novikov--Adian group. Consider the direct product $S:=N_2^r\times L(G)\in\ V$; it is finitely generated by Lemma~\ref{lm:RRW}. Denote by $J$ the ideal of $S$ consisting of all pairs whose first entry is equal to 0. Since $yb=0$ for all $y\in N_2^r$, we see that $s(b,\ovl{g})\in J$ for all $g\in G$ and all $s\in S$. Hence, if $(b,\ovl{g})$ and $(b,\ovl{h})$ are $\gJ$-related in $S$, then either $(b,\ovl{g})=(b,\ovl{h})$ or
$(b,\ovl{g})=(b,\ovl{h})t$ for some $t\in S$. Since $\ovl{h}x=\ovl{h}$ for every $x\in L(G)$, we see that either $(b,\ovl{h})t\in J$ or $(b,\ovl{h})t=(b,\ovl{h})$. Thus, $(b,\ovl{g})$ and $(b,\ovl{h})$ can only be $\gJ$-related provided that $\ovl{g}=\ovl{h}$, whence  $S$ has at least $|G|$ different $\gJ$-classes and $\V$ is not locally $\gJ$-finite, a contradiction.
\end{proof}

\begin{Proposition}
\label{prop:jlf1}
Let $\V$ be a semigroup variety.

\emph{(o)} If $\V$ excludes the semigroups $N_2^1$, $N_2^\ell$, and $N_2^r$, then $\V$ is locally $\gJ$-finite if and only if $\V$ consists of locally finite extensions of a completely regular ideal.

\emph{(r)} If $\V$ contains the semigroup $N_2^r$, but none of $N_2^\ell$ and $N_2^1$ and has finite axiomatic rank, then $\V$ is locally $\gJ$-finite if and only if $\V$ is locally $\gR$-finite.

($\ell$) If $\V$ contains the semigroup $N_2^\ell$, but none of $N_2^r$ and $N_2^1$ and has finite axiomatic rank, then $\V$ is locally $\gJ$-finite if and only if $\V$ is locally $\gL$-finite.
\end{Proposition}

\begin{proof}
(o) If $\V$  excludes $N_2^1$, $N_2^\ell$, and $N_2^r$, then Corollary~\ref{cor:grSideal} applies to each semigroup $S\in\V$ whence the set $\Gr S$ forms a completely regular ideal in $S$. If $\V$ is locally $\gJ$-finite, $S$ is a locally finite extension of $\Gr S$ by Lemma~\ref{lm:extension}.

Conversely, suppose that every member of $\V$ is a locally finite extension of a completely regular ideal. We take an arbitrary finitely generated semigroup $S$ in $\V$ and let $J$ be a completely regular ideal of $S$ such that the Rees quotient $S/J$ is locally finite. Since the semigroup $S/J$ is finitely generated, it is finite. By Lemma~\ref{lm:Jura}, $J$~is a finitely generated semigroup, and by Corollary~\ref{cor:cr-ljf} $J$ is $\gJ$-finite. Since $S$ is a small extension of $J$, Lemma~\ref{lm:small extension} ensures that $S$ is $\gJ$-finite as well.

\smallskip

(r) If $\V$ contains neither $N_2^\ell$ nor $N_2^1$, Lemma~\ref{lm:grSideal} applies to each semigroup $S\in\V$ whence the set $R:=\Gr S$ is a a right ideal in $S$. Then $SR$ is an ideal of $S$ that contains every idempotent of $S$. Now assume that $\V$ is locally $\gJ$-finite. Then Lemma~\ref{lm:extension} applies, yielding that $S$ is  a locally finite extension of $SR$. Since $\V$ contains the semigroup $N_2^r$, by Proposition~\ref{prop:l(g)_out} $\V$ contains none of the semigroups $L(G)$ with $G$ being a Novikov--Adian group. We see that all conditions of Theorem~\ref{thm:final_rlf} are satisfied, and therefore $\V$ is locally $\gR$-finite.

The converse statement follows from Lemma~\ref{lm:inclusions}.

\smallskip

($\ell$) Dual to (r).
\end{proof}

\subsection{Locally $\gJ$-finite varieties of periodic semigroups with central idempotents.} In view of Corollary~\ref{cor:1of3}, Proposition~\ref{prop:jlf1} provides characterizations of all locally $\gJ$-finite varieties of finite axiomatic rank, except those containing the semigroup $N_2^1$ but none of the semigroups $N_2^\ell$ and $N_2^r$. By Corollary~\ref{cor:central} varieties with the latter property consist of semigroups with central idempotents. Thus, it remains to characterize locally $\gJ$-finite varieties of semigroups with central idempotents. Such a characterization is provided by the next theorem, which, like Theorem~\ref{thm:cr}, applies to varieties of arbitrary rank.

\begin{Theorem}
\label{thm:central}
A variety $\V$ of periodic semigroups with central idempotents is  locally $\gJ$-finite if and only if all nilsemigroups in $\V$ are locally finite and $\V$ contains none of the semigroups $L^\flat(G)$, $R^\flat(G)$ where $G$ is a Novikov--Adian group.
\end{Theorem}

\begin{proof}
\emph{Necessity}. Local finiteness of nilsemigroups in $\V$ follows from Lemma~\ref{lm:nil}. Suppose that for some Novikov-Adian group $G$, the semigroup $L^\flat(G)$ lies in $\V$. We denote the elements of $L^\flat(G)$ according the convention of the proof of Lemma~\ref{lm:right-left}. We need also the semigroup $L^\flat_G(G)$. By the construction, it is the union of $G$ with the singletons $\{G\}$ (which is the set of the (left) cosets of $G$ in $G$) and $\{0\}$. We write $c$ instead of $\{G\}$ to improve readability. Then $L^\flat_G(G)=G\cup\{c,0\}$ and the multiplication in $L^\flat_G(G)$ extends that in $G$ and is such that $gc=cg=c$ for all $g\in G$ while all other products are equal to $0$. Observe that by Lemma~\ref{lm:subgroups} $L^\flat_G(G)$ is a homomorphic image of $L^\flat(G)$, and therefore, $L^\flat_G(G)$ belongs to the variety $\V$.

By Lemma~\ref{lm:nag0} we may (and will) assume that $G$ has an infinite finitely generated subgroup $H$ and an element $y$ such that $H\cap y^{-1}Hy=\{1\}$. Let $h_1,\dots,h_n$ generate $H$. Consider in the direct product $L^\flat(G)\times L^\flat_G(G)$, the subsemigroup $S$ generated by the pairs $(h_i,h_i)$, $i=1,\dots,n$, $(\ovl{1},1)$, and $(y,c)$. Clearly, $S$ lies in $\V$.
It is easy to calculate that
\[
S=\{(h,h),(\ovl{g},h),(gyh,c),(\ovl{gyh},c)\mid g,h\in H\}\cup J,
\]
where $J$ is the ideal of $S$ consisting of all pairs having a zero entry. Since $\ovl{g_1}\,\ovl{g_2}=0$ in $L^\flat(G)$ for all $g_1,g_2\in G$ and $c^2=0$ in $L^\flat_G(G)$, we conclude that $x(\ovl{gyh},c)\in J$ and $(\ovl{gyh},c)x\in J$ for every $x\in S\setminus\{(h,h)\mid h\in H\}$. Therefore, if for some $p_1,p_2,q_1,q_2\in H$, the pairs $(\ovl{p_1yp_2},c)$ and $(\ovl{q_1yq_2},c)$ are $\gJ$-related in $S$, there must exist some $g_1,g_2\in H$ such that
\[
(g_1,g_1)(\ovl{p_1yp_2},c)(g_2,g_2)=(\ovl{q_1yq_2},c).
\]
This implies that $g_1\ovl{p_1yp_2}g_2=\ovl{q_1yq_2}$ in $L^\flat(G)$. By the multiplication rules in $L^\flat(G)$, the left hand side of the latter equality is just $\ovl{g_1p_1yp_2}$. Thus, $\ovl{g_1p_1yp_2}=\ovl{q_1yq_2}$, whence $g_1p_1yp_2=q_1yq_2$ in $G$. This can be rewritten as $q_2p_2^{-1}=y^{-1}q_1^{-1}g_1p_1y$. By the choice of the group $G$, we have $H\cap y^{-1}Hy=\{1\}$, and therefore, $q_2p_2^{-1}=1$, that is, $p_2=q_2$. This means that $(\ovl{p_1yp_2},c)$ and $(\ovl{q_1yq_2},c)$ cannot be $\gJ$-related in $S$ whenever $p_2\ne q_2$. We see that $S$ has at least $|H|$ different $\gJ$-classes, and since the subgroup $H$ is infinite, the finitely generated semigroup $S\in\V$ fails to be $\gJ$-finite, a contradiction.

A dual argument shows that semigroups of the form $R^\flat(G)$ with $G$ being a Novikov--Adian group cannot lie in $\V$.

\smallskip

\emph{Sufficiency}. Take an arbitrary finitely semigroup $S\in\V$. Denote by $I$ the ideal of $S$ generated by the set $E(S)$ of all idempotents of $S$. Since each element in $S$ has an idempotent power, the Rees quotient $S/I$ is a nilsemigroup. Clearly, $S/I$ is finitely generated, and since all nilsemigroups in $\V$ are locally finite, $S/I$ is finite. By Lemma~\ref{lm:Jura} $I$ is a finitely generated semigroup, and by Lemma~\ref{lm:small extension} $S$ is $\gJ$-finite whenever $I$ is. Thus, it suffices to show that $I$ is $\gJ$-finite; in other words, we may assume that $S=I$, that is, $S$ is generated by the set $E(S)$ as an ideal. Since idempotents of $S$ lie in its center, we then have $S=\bigcup_{e\in E(S)}eS$. Observe that the map $x\mapsto ex$ is a homomorphism of $S$ onto $eS$ whence each subsemigroup $eS$ is finitely generated.

Clearly, the set $E(S)$ forms a semilattice. As shown by Kolibiarov\'a~\cite[Pozn\'am\-ka~1]{Kol59}, the map that assigns to each element $s$ of a given periodic semigroup with central idempotents its idempotent power $s^\omega$ is a homomorphism. Hence the semilattice $E(S)$ is a homomorphic image of $S$. Since $S$ is finitely generated, so is $E(S)$, and, clearly, every  finitely generated semilattice is finite. Thus, the set $E(S)$ is finite. We will show that $S=\bigcup_{e\in E(S)}eS$ is $\gJ$-finite by induction on $|E(S)|$.

If $|E(S)|=1$, then it is easy to see that $S$ is a group, and hence, $S$ has only one $\gJ$-class.

Suppose that $|E(S)|>1$, and let $e$ be a maximal element of the semilattice $E(S)$. Consider $T:=\bigcup_{f\in E(S)\setminus\{e\}}fS$. Clearly, $T$ is an ideal of $S$, and the Rees quotient $S/T$ has only two idempotents: $e$, which is the identity element of $S/T$, and $\{T\}$, which is the zero of $S/T$. On the other hand, $T$ is a finitely generated semigroup (since it is the union of finitely many finitely generated semigroups $fS$); besides that, it is easy to see that $T=\bigcup_{f\in E(S)\setminus\{e\}}fT$ and $E(T)=E(S)\setminus\{e\}$ whence the induction assumption applies to $T$. Hence $T$ is $\gJ$-finite. Suppose for a moment that the semigroup $S/T$ is $\gJ$-finite as well\footnote{Observe that we are \textbf{not} in a position to apply the induction assumption to $S/T$.}. If $J$ is an arbitrary ideal of $S$, we can decompose it as
\begin{equation}
\label{eq:decomposition2}
J=(J\cap(S\setminus T))\cup(J\cap T).
\end{equation}
Clearly, $J\cap T$ is an ideal of $T$ and $(J\cap(S\setminus T))\cup\{\{T\}\}$ is an ideal of $S/T$. By Lemma~\ref{lm:ideals}, if $T$ and $S/T$ are $\gJ$-finite, each of these semigroups has finitely many ideals whence there are only finitely many possibilities to choose the parts in the right hand side of the decomposition~\eqref{eq:decomposition2}, whence $S$ has finitely many ideals and so $S$ is $\gJ$-finite by Lemma~\ref{lm:ideals}.

The argument of the preceding paragraph shows that it remains to verify that the semigroup $S/T$ is $\gJ$-finite; observe that $S/T$ is finitely generated. Thus, we may assume that $S=S/T$; in other words, for the rest of the proof we assume that $S$ is a finitely generated periodic semigroup with exactly two idempotents: 1 (identity element) and 0 (zero). Clearly, the semigroup $S$ is the union of the  group $G:=\{s\in S\mid s^\omega=1\}$ with the nilsemigroup $N:=\{s\in S\mid s^\omega=0\}$, the latter being an ideal of $S$. By Lemma~\ref{lm:fg} the group $G$ is finitely generated. If $G$ is finite, then by Lemma~\ref{lm:Jura} $N$ is a finitely generated semigroup, and since all nilsemigroups in $\V$ are locally finite, $N$ is finite. Thus, in this case $S=G\cup N$ is finite, and hence, $\gJ$-finite. Therefore we may assume that $G$ is infinite, that is, $G$ is a Novikov-Adian group.

Let $X$ be a finite generating set of $S$. Each element $s\in N$ can be represented as an alternating product
\[
s=g_0a_1g_1a_2\cdots g_{n-1}a_ng_n,
\]
for certain $g_0,g_1,\dots,g_n\in G$ and $a_1,\dots,a_n\in X\cap N$. Therefore, $N$ is generated as a subsemigroup by the union of all sets of the form $Ga$ and $aG$ where $a$ runs over $X\cap N$. If each of these sets is finite, $N$ is a finitely generated semigroup, and since all nilsemigroups in $\V$ are locally finite, $N$ is finite. Clearly, the group $G$ forms a $\gJ$-class of $S$ while all other $\gJ$-classes are contained in $N$. Hence, finiteness of $N$ implies $\gJ$-finiteness of $S$.

Now we analyze the remaining option: there exists an element $a\in X\cap N$ such that at least one on the sets $Ga$ and $aG$ is infinite. We aim to show that in this situation $\V$ would contain one of the semigroups $L^\flat(\ovl{G})$ or $R^\flat(\ovl{G})$ where $\ovl{G}$ is a Novikov-Adian group. This would contradict the condition of our theorem, and therefore, the option cannot occur.

Consider the case when $Ga$ is infinite; the case when $aG$ is infinite follows by duality. Our tactics is similar to the one used in the proof of Theorem~\ref{thm:cr}: we gradually ``improve'' our semigroup by passing to its suitable divisors until we reach a semigroup isomorphic to a semigroup of the form $L^\flat_H(G)$ where $H$ is an infinite index subgroup of $G$. The desired conclusion then follows from Proposition~\ref{prop:reduction}.

Let $S_1$ be the subsemigroup of $S$ generated by $G$ and $a$. Clearly, $S_1=G\cup N_1$, where $N_1\subseteq N$ consists of alternating products of the form
\begin{equation}
\label{eq:sandwich}
s=g_0ag_1a\cdots g_{n-1}ag_n,
\end{equation}
where $g_0,g_1,\dots,g_n\in G$. Denote by $J_2$ the set of all elements that are representable as products of the form \eqref{eq:sandwich} with $n\ge 2$. Obviously, $J_2$ is an ideal of $S_1$. Suppose that $a\in J_2$, that is, $a=xayaz$ for some $x,y,z\in S_1$. Substituting the expression in the right hand side of this equality for the first occurrence of $a$ in this expression, we obtain $a=x(xayaz)yaz=x^2a(yaz)^2$. Repeating the trick, we get
\[
a=x(xayaz)yaz=x^2a(yaz)^2=\dots=x^ka(yaz)^k
\]
for every $k$, whence $a=0$ as $(yaz)^k=0$ for $k$ large enough (recall that $N$ is a nilsemigroup). This contradicts the assumption that the set $Ga$ is infinite. Thus, $a\notin J_2$, and we can pass to the Rees quotient $S_2:=S_1/J_2=G\cup N_2$, where $N_2:=N_1/J_2$. Since $a\ne 0$ in $S_2$ while every product in which $a$ occurs twice equals 0 in $S_2$, we conclude that $N_2=GaG\cup\{0\}$ is an infinite zero semigroup.

Now consider the equivalence $\rho$ on the semigroup $S_2$ such that $\rho$-classes are the singletons $\{0\}$, $\{g\}$ for $g\in G$, and the sets of the form $gaG$ where $g\in G$. It is routine to verify that $\rho$ is a congruence on $S_2$. We denote the image of $a$ in the quotient semigroup $S_3:=S_2/\rho$ by $\ovl{a}$ and identify the singleton $\rho$-classes with their elements. Then $S_3=G\cup N_3$, where $N_3=G\ovl{a}\cup\{0\}$ is an infinite zero semigroup and $\ovl{a}g=\ovl{a}$ for every $g\in G$. Let $H:=\{h\in G\mid h\ovl{a}=\ovl{a}\}$. Clearly, $H$ is a subgroup of $G$ and, for $g_1,g_2\in G$, the equality $g_1\ovl{a}=g_2\ovl{a}$ is equivalent to the inclusion $g_2^{-1}g_1\in H$ which in turn is equivalent to the equality $g_1H=g_2H$ between left cosets of $H$ in $G$. Hence the elements of $G\ovl{a}$ are in a 1-1 correspondence with the left cosets of $H$ in $G$ which implies that $H$ has infinite index in $G$. Combining the 1-1 correspondence with the identical map of $G\cup\{0\}$ onto itself, we obtain an isomorphism between $S_3$ and the semigroup $L^\flat_H(G)$. Thus, $L^\flat_H(G)\in\V$, and this completes the proof.
\end{proof}

\begin{Remark}
\label{rm:right-left}
Corollary~\ref{cor:same_variety} shows that one could have halved the set of ``forbidden objects'' in Theorem~\ref{thm:central} by restricting it to the semigroups $L^\flat(G)$, where $G$ is a Novikov--Adian group. We also observe that even though $L^\flat(G)$ and $R^\flat(G)$ serve as ``forbidden objects'' for locally $\gJ$-finite varieties, these semigroups themselves are $\gJ$-finite (actually, each of these semigroups has only three $\gJ$-classes). That is why the proof of necessity becomes non-trivial in this case.  
\end{Remark}

\subsection{Characterization.} Summing up the preparatory results obtained so far, we arrive at our ultimate characterization of locally $\gJ$-finite semigroup varieties of finite axiomatic rank.

\begin{Theorem}
\label{thm:final_jlf}
A semigroup variety $\V$ of finite axiomatic rank is locally $\gJ$-finite if and only if $\V$ is either locally finite or satisfies one of the following conditions:

\emph{(o)} $\V$ consists of locally finite extensions of periodic completely regular semigroups;

\emph{(r)} every semigroup $S\in\V$ is a locally finite extension of an ideal of the form $SR$, where $R$ is a periodic completely regular right ideal in $S$, and $\V$ contains none of the semigroups $L(G)$, where $G$ is a Novikov--Adian group;

($\ell$) every semigroup $S\in\V$ is a locally finite extension of an ideal of the form $LS$, where $L$ is a periodic completely regular left ideal in $S$, and $\V$ contains none of the semigroups $R(G)$, where $G$ is a Novikov--Adian group;

\emph{(c)} $\V$ consists of periodic semigroups with central idempotents, all nilsemigroups in $\V$ are locally finite, and $\V$ contains none of the semigroups $L^\flat(G)$, $R^\flat(G)$ where $G$ is a Novikov--Adian group.
\end{Theorem}

\begin{proof}
\emph{Necessity}. Suppose that $\V$ is locally $\gJ$-finite but not locally finite. By Corollary~\ref{cor:central}, $\V$ then contains at most one of the semigroups $N_2^1$, $N_2^\ell$, and $N_2^r$.

If $\V$ contains none of these semigroups, then $\V$ consists of locally finite extensions of a completely regular ideal by Proposition~\ref{prop:jlf1}(o).

If $\V$ contains the semigroup $N_2^r$, but none of $N_2^\ell$ and $N_2^1$, then Proposition~\ref{prop:jlf1}(r) applies whence $\V$ is locally $\gR$-finite. Using the characterization of the locally $\gR$-finite varieties of finite axiomatic rank provided by Theorem~\ref{thm:final_rlf}, we obtain that every semigroup $S\in\V$ is a locally finite extension of an ideal of the form $SR$, where $R$ is a periodic completely regular right ideal in $S$, and $\V$ contains none of the semigroups $L(G)$, where $G$ is a Novikov--Adian group.

Dually, if $\V$ contains the semigroup $N_2^\ell$, but none of $N_2^r$ and $N_2^1$, then Proposition~\ref{prop:jlf1}($\ell$) and the dual of Theorem~\ref{thm:final_rlf} imply that every semigroup $S\in\V$ is a locally finite extension of an ideal of the form $LS$, where $L$ is a periodic completely regular left ideal in $S$, and $\V$ contains none of the semigroups $R(G)$, where $G$ is a Novikov--Adian group.

Finally, if $\V$ contains the semigroup $N_2^1$, but none of $N_2^r$ and $N_2^\ell$, then $\V$ consists of semigroups with central idempotents, and Theorem~\ref{thm:central} applies.

\smallskip

\emph{Sufficiency}. If $\V$ is locally finite, nothing is to prove. If $\V$ satisfies one of the conditions (o), (r), ($\ell$), its local $\gJ$-finiteness readily follows from Proposition~\ref{prop:jlf1}, Theorem~\ref{thm:final_rlf}, and the dual of the latter. If $\V$ satisfies the condition (c), $\V$ is locally  $\gJ$-finite by Theorem~\ref{thm:central}.
\end{proof}

\section{Concluding remarks}
\label{sec:final}

\subsection{The role of finiteness of axiomatic rank}
\label{subsec:rank}
Our main results (Theorems~\ref{thm:final_hlf}, \ref{thm:final_rlf}, and \ref{thm:final_jlf}) deal with varieties of finite axiomatic rank.  A more careful analysis of the proofs shows that finiteness of axiomatic rank has only been used in the ``only if'' parts of the reduction steps of these results, that is, in the reduction of the problem of characterizing locally $\gH$-finite, locally $\gR$-finite and locally $\gL$-finite varieties to the case of varieties of periodic completely regular semigroups as well as in the reduction of the problem of characterizing locally $\gJ$-finite varieties to the case of varieties of periodic semigroups with central idempotents. Neither our characterization of
locally $\gH$-finite, locally $\gR$-finite and locally $\gL$-finite varieties of periodic completely regular semigroups (Theorem~\ref{thm:cr}) nor our characterization of locally $\gJ$-finite varieties of periodic semigroups with central idempotents (Theorem~\ref{thm:central}) requires varieties under consideration to be of finite axiomatic rank.

It would be tempting to get rid of the restriction to axiomatic rank also in the ``only if'' parts of the reduction steps but this task does not seem feasible with the present proof techniques. Indeed, these parts depend on Sapir's characterization of locally finite varieties of finite axiomatic rank (Proposition~\ref{prop:sapir}), for which proof
finiteness of axiomatic rank is crucially essential. It is an open problem whether or not Sapir's characterization extends to arbitrary varieties of periodic semigroups; if it does, so do also the main results of the present paper.

We conclude this subsection with an example demonstrating that the above discussion is non-void, so to speak; that is, we show that locally $\K$-finite varieties of infinite axiomatic rank which are not locally finite do exist for every $\K\in\{\gJ,\gD,\gL,\gR,\gH\}$. Recall that the \emph{product} of two group varieties is defined as the class of all groups $G$ possessing a normal subgroup $N$ from the first factor such that the quotient group $G/N$ lies in the second factor. The product of group varieties $\V$ and $\W$ is denoted by $\mathbf{VW}$ and is known to be a variety, see, e.g., \cite[Chapter~2]{Neu67}.

\begin{Example}
\label{ex:bpbp}
Let $p$ be a sufficiently large prime number, say, $p>10^{10}$. It is shown in \cite[Theorem 1.2]{BoOl15} that the square $\B_p\B_p$ of the Burnside variety $\B_p$ has infinite axiomatic rank and is not locally finite. On the other hand, since $\B_p\B_p$ consists of groups, this variety is locally $\K$-finite for every $\K\in\{\gJ,\gD,\gL,\gR,\gH\}$. If one wants to have a ``proper'' semigroup variety (that is, not consisting entirely of groups) with the same properties, one can consider the join of the variety $\B_p\B_p$ with the variety of zero semigroups. It can be easily verified that this join also is a locally $\K$-finite but not locally finite variety of infinite axiomatic rank.
\end{Example}

\subsection{Pure ``forbidden objects'' characterizations}
\label{subsec:alternative}
All of our main results have been stated in a mixed language: we have listed some ``forbidden objects'' but also added a few structural restrictions. In fact, it is possible to express our results in terms of ``forbidden objects'' only. For this, we need, first of all, to exhibit ``forbidden objects'' for the property which is --- explicitly or implicitly --- present in a majority of the characterizations above, namely, the property that the nilsemigroups in a given variety are locally finite. We derive such objects from  a core combinatorial construction used by Sapir~\cite{Sap87}, see also \cite[Sections 2.5 and 3.3]{Sap14}.

Fix a positive integer $k$ and let $r=6k+2$. Consider the $r^2\times r$-matrix $M$ shown in the left hand part of Figure~\ref{fig:matrices}.
\begin{figure}[ht]
\[
M:=\begin{pmatrix}
  1 & 1 & \cdots & 1 & 1\\
  \vdots & \vdots & \ddots & \vdots & \vdots \\
  1 & r & \cdots & 1 & r\\
  2 & 1 & \cdots & 2 & 1\\
  \vdots & \vdots & \ddots & \vdots & \vdots\\
  2 & r & \cdots & 2 & r\\
  \vdots & \vdots & \ddots & \vdots & \vdots\\
  r & 1 & \cdots & r & 1\\
  \vdots & \vdots & \ddots & \vdots & \vdots\\
  r & r & \cdots & r & r
\end{pmatrix} \qquad
 M_{A}:=\begin{pmatrix}
  a_{11} & a_{12} & \cdots & a_{1r-1} & a_{1 r} \\
  \vdots & \vdots & \ddots & \vdots & \vdots \\
  a_{11} & a_{r2} & \cdots & a_{1r-1} & a_{rr}\\
  a_{21} & a_{12} & \cdots & a_{2r-1} & a_{1r}\\
  \vdots & \vdots & \ddots & \vdots & \vdots \\
  a_{21} & a_{r2} & \cdots & a_{2r-1} & a_{rr}\\
  \vdots & \vdots & \ddots & \vdots & \vdots \\
  a_{r1} & a_{12} & \cdots & a_{rr-1} & a_{1r}\\
  \vdots & \vdots & \ddots & \vdots & \vdots \\
  a_{r1} & a_{r2} & \cdots & a_{rr-1} & a_{rr}
\end{pmatrix}
\]
\caption{The matrices $M$ and $M_A$}\label{fig:matrices}
\end{figure}
All odd columns of $M$ are identical and equal to the transpose of the row $(1,\dots,1,2,\dots,2,\dots,r,\dots,r)$ where each number occurs $r$ times. All even columns of $M$ are identical and equal to the transpose of the row $(1,2,\dots,r,\dots,1,2,\dots,r)$ in which the block $1,2,\dots,r$ occurs $r$ times.

Now consider the alphabet $A=\{a_{ij}\mid 1\leq i,j\leq r\}$ of cardinality $r^2$. We convert the matrix $M$ to the matrix $M_{A}$ (shown in the right hand part of Figure~\ref{fig:matrices}) by replacing numbers with letters according to the following rule: whenever the number $i$ occurs in the column $j$ of $M$, we substitute it with
the letter $a_{ij}$.

Let $v_{t}$ be the word in the $t^{\mathrm{th}}$ row of the matrix $M_A$. Consider the endomorphism $\gamma$ of the free semigroup $A^+$ over the alphabet $A$ defined by
\[
\gamma(a_{ij}):=v_{(i-1)r+j}.
\]
Let $V_k$ be the set of all factors of the words in the sequence $\{\gamma^m(a_{11})\}_{m=1,2,\dots}$ and let $0$ be a fresh symbol beyond $V_k$. We define a multiplication $\cdot$ on the set $V_k\cup\{0\}$ as follows:
\[
u\cdot v:=\begin{cases}
uv& \text{if}\ uv\in V_k,\\
0 & \text{otherwise}.
\end{cases}
\]
Clearly, the set $V_k^0:=V_k\cup\{0\}$ becomes a semigroup under this multiplication. This semigroup has been introduced in~\cite{ADV12} for a different purpose.

Let $\V$ be a variety of finite axiomatic rank. We say that the rank of $\V$ is equal to $k$ and write $\rk\V=k$ if $k$ is the least number such that $\V$ can be given by a set of identities involving at most $k$ letters. The following is a reformulation of Sapir's characterization of locally finite varieties of finite axiomatic rank consisting of nilsemigroups (that we mentioned in the introduction) in the language of ``forbidden objects''.
\begin{Proposition}
\label{prop:avoidable}
Let $\V$ be a semigroup variety with $\rk\V=k$. All nilsemigroups in $\V$ are locally finite if and only if $\V$ excludes the semigroup $V_k$.
\end{Proposition}

\begin{proof}
\emph{Necessity}.  By the construction, the semigroup $V_k^0$ is infinite and is generated by the finite set $A$. It follows from \cite[Theorem 2.5.19]{Sap14} that no word in the sequence $\{\gamma^m(a_{11})\}_{m=1,2,\dots}$ contains a factor of the form $uu$, where $u\in A^+$. This implies that $u\cdot u=0$ in $V_k^0$ for every $u\in V_k$ whence $V_k$ is a nilsemigroup. Thus, $V_k^0$ cannot belong to any variety whose finitely generated nilsemigroups are finite.

\smallskip

\emph{Sufficiency}. Recall that the sequence $\{Z_n\}_{n=1,2,\dots}$ of \emph{Zimin words} is defined inductively by $Z_1:=x_1$, $Z_{n+1}:=Z_nx_{n+1}Z_n$. The contrapositive of \cite[Lemma~3.3.34]{Sap14} implies that if $\rk\V=k$ and $\V$ excludes the semigroup $V_k$, then $\V$ satisfies an identity of the form $Z_{k+1}=z$, where $z$ is a word different from $Z_{k+1}$. By \cite[Theorem~3.3.4]{Sap14}, this ensures that all  nilsemigroups in $\V$ are locally finite.
\end{proof}

Now we are ready to characterize locally $\K$-finite semigroup varieties of finite axiomatic rank exclusively in terms of ``forbidden objects''. In Table~\ref{tab:fo} we collect such characterizations for an arbitrary variety $\V$ with $\rk\V=k$. The acronym NAG stands for ``Novikov--Adian group''.

\begin{table}[h]
\caption{Characterizations of locally $\K$-finite varieties of rank $k$}\label{tab:fo}
\begin{tabular}{|p{3cm}|p{8cm}|}
  \hline
  \textbf{Property}\rule[-6pt]{0pt}{18pt} & \textbf{``Forbidden objects''} \\
  \hline
  Local $\gH$-finiteness\rule[-6pt]{0pt}{18pt} & $V_k^0$, $N_2^1$, $N_2^\ell$, $N_2^r$, $L(G)$, $R(G)$, where $G$ is a NAG \\
  \hline
  Local $\gR$-finiteness\rule[-6pt]{0pt}{18pt} & $V_k^0$, $N_2^1$, $N_2^\ell$, $L(G)$, where $G$ is a NAG \\
  \hline
  Local $\gL$-finiteness\rule[-6pt]{0pt}{18pt} & $V_k^0$, $N_2^1$, $N_2^r$, $R(G)$, where $G$ is a NAG \\
  \hline
  Local $\gJ$-finiteness\rule[-6pt]{0pt}{18pt} & Either $V_k^0$, $N_2^1$, $N_2^\ell$, $N_2^r$,\\
  \rule[-6pt]{0pt}{18pt}& or $V_k^0$, $N_2^1$, $N_2^\ell$, $L(G)$, where $G$ is a NAG,\\
  \rule[-6pt]{0pt}{18pt}& or $V_k^0$, $N_2^1$, $N_2^r$, $R(G)$, where $G$ is a NAG,\\
  \rule[-6pt]{0pt}{18pt}& or $V_k^0$, $\mathbb{N}$, $N_2^\ell$, $N_2^r$, $L^\flat(G)$, $R^\flat(G)$, where $G$ is a NAG\\
  \hline
\end{tabular}
\end{table}

The results gathered in Table~\ref{tab:fo} readily follow from the proofs of the main theorems of the present paper (Theorems~\ref{thm:final_hlf}, \ref{thm:final_rlf}, and \ref{thm:final_jlf}), combined with Proposition~\ref{prop:avoidable} and lemmas in Subsection~\ref{subsec:3-element} that clarify the role of the 3-element semigroups $N_2^1$, $N_2^\ell$,  and $N_2^r$. Therefore we include no proofs of these results here. Perhaps, a word of comment is needed to explain the sudden appearance of $\mathbb{N}$, the additive semigroup of positive integers, in the last line of Table~\ref{tab:fo}: one has to exclude $\mathbb{N}$ in order to guarantee that the variety under consideration is periodic. (In all other cases, periodicity follows from the condition that the semigroup $N_2^1$ is excluded and from the obvious fact that $N_2^1$ lies in the variety $\var\mathbb{N}$.)

What can be said about the size of the lists of ``forbidden objects'' in Table~\ref{tab:fo}? It is known \cite{Koz12} that the variety $\B_p$ where $p$ is a sufficiently large prime number contains uncountably many subvarieties each of which is generated by a Novikov--Adian group. Combining this result with Lemma~\ref{lm:difference} and its dual, we see that the lists of ``forbidden objects'' in Table~\ref{tab:fo} are uncountable except for the first list in the cell corresponding to locally $\gJ$-finite varieties.

\subsection{Local $\K$-finiteness vs. $\K$-compatibility}
\label{subsec:compatibility}
The reader acquainted with the cycle of papers \cite{PaVo05,PaVo06a,PaVo06b} devoted to the study of varieties consisting of $\K$-\emph{compatible} semigroups, i.e., semigroups in which Green's relations are congruences, may have observed that the present paper has borrowed several ideas and technical tools from this cycle. We did not emphasize the relation of our results to those in \cite{PaVo05,PaVo06a,PaVo06b} wherever it would be in place as we did not want to deviate from the main topic of the present paper and aimed to make our proofs as self-contained as possible. Here we only want to register a corollary that looks a bit surprising but in fact readily follows from a comparison between the main results of the present paper and the characterizations of $\K$-compatible varieties found in~\cite{PaVo05,PaVo06a,PaVo06b}.

\begin{Corollary}
\label{cor:compatibility}
Let $\K$ be any of the Green relations $\gJ,\gD,\gL,\gR,\gH$. If all nilsemigroups in a periodic semigroup variety $\V$ are locally finite and $\K$ is a congruence on every semigroup in $\V$, then the variety $\V$ is locally $\K$-finite.\qed
\end{Corollary}

It should be noticed that, in general, a variety consisting of $\K$-compatible semigroups need not be periodic nor need its nilsemigroups be locally finite. The converse of Corollary~\ref{cor:compatibility} also does not hold true in general.

\subsection{Further developments and open questions}
\label{subsec:questions}
We plan to extend the studies initiated in this paper to several important species of semigroups that are naturally equipped with a unary operation: completely regular semigroups, inverse semigroups, epigroups. In each of these cases, the varietal approach is well established (see respectively the monographs \cite{PR99} and \cite{Pet84}, and the survey \cite{She05}) and Green's relations are of immense importance so that looking for characterizations of locally $\K$-finite varieties of completely regular semigroups, inverse semigroups, and epigroups appears to be worthwhile. A more general problem is to characterize locally $\K$-finite e-varieties of regular semigroups. (Recall that a class of regular semigroups is said to be an \emph{existence variety} (or \emph{e-variety}) if it is closed under taking direct products, regular subsemigroups and homomorphic images; this notion was introduced by Hall \cite{hall1} and, independently, by Kad\'ourek and Szendrei \cite{kad-szen} for the class of orthodox semigroups.) The interconnections between local $\K$-finiteness and $\K$-compatibility mentioned in Subsection~\ref{subsec:compatibility} give some hope that a characterization might be possible even in such an extremely general setting: we mean here that there exists rather a transparent characterization of $\gH$-compatible e-varieties of regular semigroups in the language of ``forbidden objects'', see \cite{PaVo96}.

Finally, a challenging open problem is to characterize varieties whose finitely generated semigroups have only finitely many \textbf{regular} $\K$-classes. This condition imposes no restriction to nilsemigroups, nor it implies periodicity. For $\K=\gH$, it amounts to saying that each finitely generated semigroup in the variety under consideration has only finitely many idempotents which is a fairly natural finiteness condition, certainly deserving being explored. It is interesting to notice that this condition has recently been studied for varieties of associative rings, see \cite{Fin15}.


\begin{thebibliography}{12}
\bibitem{Adi79}
\textsc{S.~I. Adian}, ``The {Burnside} Problem and Identities in Groups'', Springer Verlag, 1979.

\bibitem{Asc04}
\textsc{M. Aschbacher}, \emph{The status of the Classification of the Finite Simple Groups}, Notices Amer.\ Math.\ Soc. \textbf{51} (2004), 736--740.

\bibitem{ADV12}
\textsc{K. Auinger}, \textsc{I. Dolinka}, \textsc{M. V. Volkov}, \emph{Matrix identities involving multiplication and transposition}, J. European Math. Soc. \textbf{14} (2012), 937--969.

\bibitem{BoOl15}
\textsc{N. S. Boatman}, \textsc{A. Yu. Olshanskii}, \emph{On identities in the products of group varieties},  Internat. J. Algebra Comput. \textbf{25} (2015), 531--540.

\bibitem{BuSa81}
\textsc{S. Burris}, \textsc{H. P. Sankappanavar}, ``A Course in Universal Algebra'', Springer Verlag, 1981.

\bibitem{CaMa13}
\textsc{A. J. Cain}, \textsc{V. Maltcev}, \emph{For a few elements more: A survey of finite Rees index}, Preprint, 2013 \url{https://arxiv.org/abs/1307.8259}	

\bibitem{ClPr61}
\textsc{A. H. Clifford}, \textsc{G. B. Preston} ``The Algebraic Theory of Semigroups'', Vol.I, 2nd ed., Amer.\ Math.\ Soc., 1964.

\bibitem{Gre51}
\textsc{J. A. Green}, \emph{On the structure of semigroups}, Ann. Math. (2) \textbf{54} (1951) 163--172.

\bibitem{Fin15}
\textsc{O. B. Finogenova}, \emph{Characterization of locally Noetherian varieties in terms of idempotents}, Math. Notes \textbf{97} (2015), 937-–940.

\bibitem{HaHi56}
\textsc{P. Hall}, \textsc{G. Higman}, \emph{On the $p$-length of $p$-soluble groups and reduction theorems for Burnside's problem}. Proc. London Math. Soc. (3) \textbf{6} (1956), 1--42.

\bibitem{hall1}
\textsc{T. E. Hall}, \emph{Identities for existence varieties of regular semigroups}, Bull. Austral. Math. Soc. \textbf{40} (1989) 59--77.

\bibitem{How95}
\textsc{J. M. Howie}, ``Fundamentals of Semigroup Theory'', 2nd ed., Clarendon Press, 1995.

\bibitem{How02}
\textsc{J. M. Howie}, \emph{Semigroups, past, present and future}, in Wanida Hemakul (ed.), ``Proc. Internat. Conference on Algebra and its Applications'', Chulalongkorn Univ., Bangkok, Thailand, 2002, 6--20.

\bibitem{Iva94}
\textsc{S.~V. Ivanov}, \emph{The {Burnside} problem for all sufficiently large exponents}, Internat. J. Algebra Comput. \textbf{4} (1994), 1--300.

\bibitem{Jur78}
\textsc{A. Jura}, \emph{Determining ideals of a given finite index in a finitely presented semigroup}, Demonstratio Math. \textbf{11} (1978), 813--827.

\bibitem{kad-szen}
\textsc{J. Kad\'ourek}, \textsc{M. B. Szendrei}, \emph{A new approach in the theory of orthodox semigroups}, Semigroup Forum \textbf{40} (1990) 257--296.

\bibitem{KoWa57}
\textsc{R. J. Koch}, \textsc{A. D. Wallace}, \emph{Stability in semigroups}, Duke Math. J. \textbf{24} (1957), 193--195.

\bibitem{Kol59}
\textsc{B. Kolibiarov\'a}, \emph{O \v{c}iasto\v{c}ne komutat\'\i{}vnych periodick\'ych pologrup\'ach}, Matematicko-Fyzik\'alny \v{C}asopis \textbf{9}, No. 3 (1959), 160--172.

\bibitem{Koz12}
\textsc{P. A. Kozhevnikov}, \emph{On nonfinitely based varieties of groups of large prime exponent}, Comm. Algebra \textbf{40} (2012), 2628--2644.

\bibitem{Lys96}
\textsc{I. G. Lysenok}, \emph{Infinite Burnside groups of even exponent}, Izvestiya: Mathematics \textbf{60} (1996), 453--654.

\bibitem{Neu67}
\textsc{H.~Neumann}, ``Varieties of Groups'', Springer Verlag, 1967.

\bibitem{NoAd68}
\textsc{P. S. Novikov}, \textsc{S. I. Adian}, \emph{On infinite periodic groups}. I, II, III. Math. USSR--Izv. \textbf{2} (1968), 209--236, 241--479, 665--685.

\bibitem{PaVo96}
\textsc{F. Pastijn}, \textsc{M. V. Volkov}, \emph{Minimal non-cryptic e-varieties of regular semigroups}, J. Algebra \textbf{184} (1996) 881--896.

\bibitem{PaVo05}
\textsc{F. Pastijn}, \textsc{M. V. Volkov}, \emph{$R$-compatible varieties of semigroups}, Acta Sci.\ Math. (Szeged) \textbf{71} (2005), 521--554.

\bibitem{PaVo06a}
\textsc{F. Pastijn}, \textsc{M. V. Volkov}, \emph{Cryptic varieties of semigroups}, Semigroup Forum \textbf{72} (2006), 159--189.

\bibitem{PaVo06b}
\textsc{F. Pastijn}, \textsc{M. V. Volkov}, \emph{$D$-compatible varieties of semigroups}, J. Algebra \textbf{299} (2006), 62--93.

\bibitem{Pet84}
\textsc{M. Petrich}, ``Inverse Semigroups'',  John Wiley \& Sons, Inc., 1984.

\bibitem{PR99}
\textsc{M. Petrich}, \textsc{N. R. Reilly}, ``Completely Regular Semigroups'',  John Wiley \& Sons, Inc., 1999.

\bibitem{Ras81}
\textsc{V. V. Rasin}, \emph{On the varieties of Cliffordian semigroups}, Semigroup Forum \textbf{23} (1981), 201--220.

\bibitem{RRW98}
\textsc{E. F. Robertson}, \textsc{N. Ru\v{s}kuc}, \textsc{J. Wiegold}, \emph{Generators and relations of direct products of semigroups}, Trans. Amer.\ Math.\ Soc. \textbf{350} (1998), 2665--2685.

\bibitem{Rob96}
\textsc{D. J. S. Robinson}, ``A Course in the Theory of Groups'', Springer Verlag, 1996.

\bibitem{Rus98}
\textsc{N. Ru\v{s}kuc}, \emph{On large subsemigroups and finiteness conditions of semigroups}, Proc. London Math. Soc. (3) \textbf{76} (1998), 383--405.

\bibitem{Sap87}
\textsc{M.~V.~Sapir}, \emph{Problems of Burnside type and the finite basis property in varieties of semigroups}, Math. USSR--Izv. \textbf{30} (1987), 295--314.

\bibitem{Sap14}
\textsc{M.~V.~Sapir}, with contributions by \textsc{V. S. Guba} and \textsc{M. V. Volkov}, ``Combinatorial Algebra: Syntax and Semantics'', Springer Verlag, 2014.

\bibitem{She05}
\textsc{L. N. Shevrin}, \emph{Epigroups}, in Valery B. Kudryavtsev, Ivo G. Rosenberg, and Martin Goldstein (eds.), ``Structural Theory of Automata, Semigroups, and Universal Algebra'', Springer Verlag, 2005, 331--380.

\bibitem{SiSo16}
\textsc{P. V. Silva}, \textsc{F. Soares}, \emph{Local finiteness for Green relations in ($I$-)semigroup varieties}, Preprint, 2016 \url{https://arxiv.org/abs/1606.03866}	

\bibitem{Vol89}
\textsc{M. V. Volkov}, Semigroup varieties with modular lattice of subvarieties, Sov. Math. (Izvestiya VUZ) \textbf{33}, No.6 (1990), 48--58.

\bibitem{Vol00}
\textsc{M. V. Volkov}, \emph{``Forbidden divisor'' characterizations of epigroups with certain properties of group elements}, in ``Algebraic Systems, Formal Languages and Computations'' [Surikaisekikenkyusho Kokyuroku 1166], Research Institute for Mathematical Sciences, Kyoto University, 2000, 226--234.

\bibitem{Zel90}
\textsc{E. I. Zelmanov}, \emph{Solution of the restricted Burnside problem for groups of odd exponent}, Math. USSR--Izv. \textbf{36} (1990), 41--60.

\bibitem{Zel91}
\textsc{E. I. Zelmanov}, \emph{Solution of the restricted Burnside problem for 2-groups}, Math. USSR-Sb. \textbf{72} (1992), 543--565.
\end{thebibliography}
\end{document}